\documentclass[11pt,reqno]{amsart}

\usepackage[utf8]{inputenc}
\usepackage[margin=1.25in]{geometry}
\usepackage{hyperref}
\usepackage{appendix}
\usepackage{amsfonts}
\usepackage{amssymb}
\usepackage{stmaryrd} 
\usepackage{amsmath}
\usepackage{amsthm}
\usepackage{mathrsfs}
\usepackage{xcolor}
\theoremstyle{plain}
\newtheorem{theorem}{Theorem}[section]

\newtheorem{lemma}[theorem]{Lemma}
\newtheorem{proposition}[theorem]{Proposition}
\newtheorem{Definition}[theorem]{Definition}
\theoremstyle{remark}

\newtheorem{remark}[theorem]{Remark}

\usepackage{biblatex} %Imports biblatex package
\numberwithin{equation}{section}

\title[Ergodicity stochastic heat equation Hölder coefficient]{Exponential ergodicity of stochastic heat equations with Hölder coefficients}
\author{Yi HAN}
\address{Department of Mathematics, Massachusetts Institute of Technology, Cambridge MA.
}
\email{hanyi16@mit.edu}
\thanks{Supported by EPSRC grant EP/W524141/1 and Simons Collaboration grant (601948, DJ).}

\begin{document}

\begin{abstract}
    We investigate the stochastic heat equation driven by space-time white noise defined on an abstract Hilbert space, assuming that the drift and diffusion coefficients are both merely Hölder continuous. Random field SPDEs are covered as special examples. We give the first proof that there exists a unique in law mild solution when the diffusion coefficient is $\beta$ - Hölder continuous for $\beta>\frac{3}{4}$ and uniformly non-degenerate, and that the drift is locally Hölder continuous. Meanwhile, assuming the existence of a suitable Lyapunov function for the SPDE, we prove that the solution converges exponentially fast to the unique invariant measure with respect to a typical Wasserstein distance. Our technique generalizes when the SPDE has a Burgers type non-linearity $(-A)^{\vartheta}F(X_t)$ for any $\vartheta\in(0,1)$, where $F$ is $\vartheta+\epsilon$- Hölder continuous and has linear growth. For $\vartheta\in(\frac{1}{2},1)$ this result is new even in the case of additive noise. 
\end{abstract}

\maketitle

\section{Introduction}
In the past thirty years, there have been significant developments in the field of stochastic partial differential equations. Two classical notions of solutions, including vector field solutions \cite{carmona1986introduction} and solutions as evolution equations defined on Hilbert spaces \cite{da2014stochastic}, are used predominantly in the literature, while the new solution theory of regularity structures \cite{hairer2014theory} gained much success in solving singular SPDEs.

In most of these solution theories, it is assumed that either the noise is an additive noise, or the coefficient in front of the noise is locally Lipschitz continuous in the argument $X_t$. Not much is done concerning the following SPDE defined on a separable Hilbert space
\begin{equation}\label{heatequation}
    dX_t=\Delta X_t dt+b(X_t)dt+\sigma(X_t)dW_t
\end{equation}
when the diffusion coefficient $\sigma$ is not locally Lipschitz continuous in its argument $X_t$. In this paper we outline a comprehensive study of such SPDEs in an abstract framework, covering a wide variety of models that were previously intractable. 

This problem is interesting because it illustrates the crucial difference between finite and infinite dimensional analysis. For finite dimensional SDEs, elliptic PDE techniques (see for example \cite{krylov2021strong} and the references therein), or the simple trick of representing the solution as a time-changed Brownian motion \cite{stroock1997multidimensional}, can easily solve such problems of a non-Lipschitz, non-degenerate diffusion coefficient. When the diffusion coefficient is degenerate, one should mention Yamada and Watanabe’s $1/2$-Hölder continuity result \cite{yamada1971uniqueness}. In infinite dimensions, the story is quite different as most of the finite-dimensional techniques are no longer available. A notable progress is made by Mytnik and Perkins \cite{mytnik2011pathwise} who study 1-d vector field stochastic heat equation with Lipschitz drift and $\frac{3}{4}+\epsilon$- Hölder noise coefficient, generalizing the result of \cite{yamada1971uniqueness} to infinite dimensions. See also \cite{mueller2014nonuniqueness}, \cite{10.1214/14-AOP962} for related non-uniqueness results. In this paper we take an abstract perspective to study Hilbert space valued SPDEs. 

In this paper we adopt the following notion of solutions:
\begin{Definition}
Consider a separable Hilbert space $H$ and a densely defined operator $A$ on $H$. Let $\mathcal{S}(t)$ denote the semigroup on $H$ generated by $A$. Consider the following evolution equation on $H$ (the conditions on $b$ and $\sigma$ will be specified explicitly in the sequel),
$$
dX_t=AX_t dt+b(t,X_t)dt+\sigma(t,X_t)dW_t,\quad X_0=x\in H. 
$$ A \textbf{weak mild} solution to this equation consists of a sequence $(\Omega,\mathcal{F},(\mathcal{F}_t),\mathbb{P},W,X)$, where $(\Omega,\mathcal{F},(\mathcal{F}_t),\mathbb{P})$ is a filtered probability space defining a cylindrical Wiener process $W$ on $H$ adapted to $\mathcal{F}_t$, and $(X_t)_{t\geq 0}$ is an $H$-valued, $\mathcal{F}_t$-adapted continuous process that satisfies, $\mathbb{P}$-a.s., the following for any $t>0$:
\begin{equation}
X_t=\mathcal{S}(t)x+\int_0^t \mathcal{S}(t-s)b(s,X_s)ds+\int_0^t \mathcal{S}(t-s)\sigma(s,X_s)dW_s.    
\end{equation}
More generally, for the following SPDE of Burgers type, where $\vartheta\in(0,1)$,
$$
dX_t=AX_t dt+b(t,X_t)dt+(-A)^{\vartheta}F(t,X_t)dt+\sigma(t,X_t)dW_t,\quad X_0=x\in H,  
$$ a \textbf{weak mild} solution consists of a sequence $(\Omega,\mathcal{F},(\mathcal{F}_t),\mathbb{P},W,X)$, where 
\begin{equation}
X_t=\mathcal{S}(t)x+\int_0^t \mathcal{S}(t-s)b(s,X_s)ds+\int_0^t(-A)^{\vartheta}\mathcal{S}(t-s)F(s,X_s) ds+\int_0^t \mathcal{S}(t-s)\sigma(s,X_s)dW_s.    
\end{equation}

\end{Definition}

The reader will be aware that we consider probabilistic weak solutions rather than probabilistic strong, or pathwise, solutions to these SPDEs. This is in particular due to the infinite dimensional setting we are interested in in this paper. For strong solutions of Hilbert space valued SPDEs with irregular drift, one should mention the case of a Hölder continuous drift considered in \cite{da2010pathwise}, and the case of a merely bounded measurable drift considered in \cite{da2013strong} and \cite{da2016strong} where strong solutions are only proved for some initial conditions but not for all. These works heavily exploit the Gaussian structure of white noise, and do not seem to easily generalize to multiplicative noise setting. For these reasons, the solution theory of Hilbert space SPDEs with multiplicative Hölder continuous noise coefficient seems completely open, and even weak well-posedness is beyond reach by most existing techniques. In this paper we fill this gap and build a brand new solution theory to these SPDEs.

Although only weak well-posedness is considered in this paper, our solution theory is remarkably robust. We do not impose any structural conditions on the coefficient beyond their Hölder continuity, so they cover vector field solutions in any dimension. We can consider the solution with an additional Burgers type non-linearity, and we can study all the long-term behaviours of these SPDEs by proving their exponential convergence to the unique invariant measure under a Lyapunov condition, even though the coefficients of the SPDEs are not Lipschitz continuous. Finally, the technique generalizes to stochastic wave equation as well, see the parallel work \cite{HAN2024110224}.

Now we state the main results of this paper.
\subsection{Well-posedness of stochastic heat equation with Hölder coefficients} We begin with the well-posedness results.

Throughout the paper we will use the notation $|\cdot|$ to represent a norm, and unless specified specifically, the norm is with respect to the Hilbert space $H$ on which the evolution equation is defined. In section \ref{section240} however, we will also use this notation for the absolute value on $\mathbb{R}$ and for the Euclidean norm on $\mathbb{R}^r$- this will not cause any confusion as the meaning should be absolutely clear from the context.

\begin{theorem}\label{theorem1.1}
Consider the evolution equation 
\begin{equation}\label{evolutionequation}
    dX_t=A X_t dt+b(t,X_t)dt+\sigma(t,X_t)dW_t,\quad X_0=x\in H
\end{equation}
on a Hilbert space $H$, where $W$ is a cylindrical Wiener process on $H$. Assume that

$\mathbf{H_1}$ The densely defined operator $A:\mathcal{D}(A)\subset H\to H$
is diagonalizable on $H$, with eigenvectors $\{e_n\}_{n\geq 1}$ forming an orthonormal basis in $H$, such that the eigenvalues $$Ae_n=-\lambda_n e_n,\quad \lambda_n>0,$$ satisfy that $\lambda_n\geq c>0$ for some constant $c>0$. Moreover assume $\lambda_n$ satisfies the following summability condition: there exists some $\eta_0\in(0,1)$ such that for each $\eta\in(0,\eta_0),$
\begin{equation}\label{summableeta0}\sum_{n\geq 1}\frac{1}{{{\lambda_n}^{1-\eta}}}<\infty.\end{equation}
(The meaning of $\eta_0$, and some examples, will be explained after statement of the theorem. This $\eta_0$ crucially captures the linear contraction rate of the dynamics.)
Further, we assume there exists a Banach space $H_0\subset H$ with continuous embedding, that $\{e_n\}_n\subset H_0$, and that there exists some $M>0$ such that 
\begin{equation}\label{new1.71.7}
    \sup_n |e_n|_{H_0}\leq M<\infty.
\end{equation}
In the following we use the notation $\mathcal{L}(H_0,H)$ to denote the space of continuous linear mappings from $H_0$ to $H$.

$\mathbf{H_2}$ The drift coefficient $b(t,x):[0,\infty)\times H\to H$ is measurable and satisfies, for some $\alpha\in(0,1]$ and $M>0$ independent of $t$,
\begin{equation}
    |b(t,x)-b(t,y)|\leq M|x-y|^\alpha,\quad |x-y|\leq 1,\quad t>0.
\end{equation}

$\mathbf{H_3}$ The diffusion coefficient $\sigma(t,x):[0,\infty)\times H\to \mathcal{L}(H_0,H)$ is measurable and satisfies: for some $\beta\in(1-\frac{\eta_0}{2},1]$ and some $M>0$ independent of $t$, we have 
\begin{equation}
|\sigma(t,x)-\sigma(t,y)|_{\mathcal{L}(H_0,H)}\leq M|x-y|^\beta,\quad |x-y|\leq 1,\quad t>0.\end{equation}
In particular, for the Laplacian operator on $[0,1]$ with Dirichlet boundary condition, one can take $\eta_0=\frac{1}{2}$, and $\sigma$ needs to be $\frac{3}{4}+\epsilon$-Hölder continuous in $x$.

 We further assume that $\sigma$ is self-adjoint: for any $t>0$ and $x\in H$,
\begin{equation}
    \langle \sigma(t,x)e_j,e_k\rangle=\langle e_j,\sigma(t,x)e_k\rangle.\quad j,k\in\mathbb{N}_+.
\end{equation}

$\mathbf{H_4}$ For each $t>0$, $x\in H$ we can find a Banach space $\widetilde{H}_0(t,x)\subset H$ with the induced norm such that $H_0\subset \widetilde{H}_0(t,x)\subset H$ and there is an extension of $\sigma(t,x)$ so that $\sigma(t,x)\in \mathcal{L}(\widetilde{H}_0(t,x),H)$ such that, given any $x\in H$, $\sigma(t,x)$ has a right inverse $\sigma^{-1}(t,x)$ on $H$:
$$\sigma(t,x)\sigma^{-1}(t,x)u=u\quad \text{for any} \quad u\in H,$$
and that $\sigma^{-1}(t,x)H\subset \widetilde{H}_0(t,x)$. Moreover, we assume the following boundedness condition is satisfied: there exists some $C>0$ independent of $t$ such that 
\begin{equation}
    \sup_{x\in H}|\sigma^{-1}(t,x)|_{\mathcal{L}(H,H)}<C<\infty.
\end{equation}

$\mathbf{H_5}$ The coefficients $b$ and $\sigma$ satisfy the linear growth assumption: for some $M>0$ independent of $t$, 
\begin{equation}|b(t,x)|\leq M(1+|x|), \quad |\sigma(t,x)|_{\mathcal{L}(H_0,H)}\leq M( 1+|x|),\quad x\in H.\end{equation}

Then there exists a unique weak mild solution to \eqref{evolutionequation}. 
\end{theorem}

The constant $\eta_0$ before \eqref{summableeta0} is the most crucial parameter in Theorem \ref{theorem1.1}. It controls the speed at which eigenvalues $\lambda_i$ tends to infinity. Consider 1-d Laplacian on an interval, then we can take $\eta_0=\frac{1}{2}$, see the next paragraph. Consider 2-d Laplacian, we can only take $\eta_0=1$ which is forbidden. Thus a SPDE with (standard) Laplacian operator in 2d is not covered, but it is also well-known that such a solution is distribution-valued but not function-valued. The existence of $\eta_0\in(0,1)$ ensures that the solution can be function-valued in the vector field case, which is also necessary as we can only compose a function by a Hölder continuous mapping but not compose a distribution by such a mapping. The precise value of $\eta_0$ is also crucial as it will show up in Proposition \ref{relabel}, controlling the speed of linear contraction. The latter further restricts the range of $\beta$ in assumption $\mathbf{H}_3$.

We briefly discuss how our abstract definition of stochastic PDEs covers the \textit{random field} SPDEs introduced in the literature. Consider the 1d stochastic heat equation defined on a bounded domain $[0,1]$ and Dirichlet boundary condition. Then we take $H=L^2([0,1];\mathbb{R})$ and $A=\Delta$ is the Laplace operator on a finite interval of $\mathbb{R}$ with Dirichlet boundary conditions. The eigenvalues are $-k^2\pi^2,k\in\mathbb{N}_+$ and $\mathbf{H}_1$ is satisfied with $\eta_0=\frac{1}{2}$. The eigenvectors are the sine and cosine functions, so we can take $H_0$ to be the Banach space $\mathcal{C}([0,1];\mathbb{R})$ of bounded continuous functions, which embed continuously into $H$. The assumption \ref{new1.71.7} is trivially satisfied by the sine and cosine functions. Note however that assumption $\mathbf{H}_1$ excludes stochastic heat equation in dimensions two or higher when the solution is merely distribution valued, and that no $\eta_0\in(0,1)$ can be found to satisfy assumption $\mathbf{H}_1$.
On the other hand, if we take $A=-\Delta^2$ where $\Delta$ is the Laplacian on a compact domain in $\mathbb{R}^d,d\leq 3$ with smooth boundary and Dirichlet boundary condition, then assumption $\mathbf{H}_1$ can still be satisfied with some choice of $\eta_0$, so that $\mathbf{H}_1$ does cover some interesting SPDEs in dimensions two and higher.

Then we discuss how assumptions $\mathbf{H}_2$ to $\mathbf{H}_5$ are satisfied by random field solutions. Consider a random field SPDE on $[0,1]\subset\mathbb{R}$,
$$
\frac{\partial u}{\partial t}(t,x)=\Delta u(t,x)+B(t,x,u(t,x))+g(t,x,u(t,x))dW(t,x),\quad u(0,x)=u_0(x).
$$
Assuming that $B$ is $\alpha$-Hölder continuous in its third argument uniform in $t$, $x$, then defining for every $x\in H$ and $z\in[0,1]$, $b(t,x)(z):=B(t,z,x(z))$, we can check that $b(t,x)$ satisfies $\mathbf{H}_2$ with the same $\alpha$. Assuming that $g$ is $\beta$- Hölder continuous in its third argument uniformly in $t$ and $x$, then 
defining for every $x\in H$, $u\in H$ and $z\in[0,1]$, $\sigma(t,x)u(z)=g(t,z,x(z))u(z)$,  we can check that $\mathbf{H}_3$ is satisfied with the same $\beta>0$, and the self-adjointness condition  follows readily from our definition. 

As a technical reminder, we assume that $\sigma(t,x)$ takes place in $\mathcal{L}(H_0,H)$ (instead of $\mathcal{L}(H,H)$) whereas $W$ is a Wiener process on $H$. This is a standard setting in the definition of stochastic integrals in SPDE literature, see for example \cite{da2014stochastic}, Chapter 4 and 6. In the concrete case of vector field SPDEs we make such a restriction so that we can define $\sigma(t,x)$ more conveniently: for example, on the domain $[0,1]$, take $H_0=L^\infty([0,1])\subset H=L^2([0,1])$ and $\sigma$ is given by multiplication of $u\in H_0$ with some function $v\in H$. We can multiply a bounded function by a $L^2$ function and get a function in $L^2$, but multiplication of two $L^2$ functions is not closed in $L^2$. No generality is lost in this restriction, as eventually we will show the solution to the SPDE $X_t$ is $H_0$-valued almost surely.

Then we verify assumption $\mathbf{H}_4$: if we assume that $|g(t,x,u)|>\epsilon$ for some $\epsilon>0$ and all $t,x,u$, then we define $\sigma^{-1}:[0,\infty)\times H\to \mathcal{L}(H,H)$ through 
$$
\sigma^{-1}(t,x)u(z)=\frac{1}{g(t,z,x(z))}u(z),\quad z\in[0,1]
$$for any $x,u\in H$. Then one readily verifies (see below) that after a slight extension of $\sigma$  to be an element of $\mathcal{L}(\widetilde{H}_0(t,x),H)$ for some $H_0\subset\widetilde{H}_0(t,x)$, then we have $\sigma(t,x)\sigma^{-1}(t,x)u=u$ for any $u\in H$, any $t>0,x\in H$. Our assumption on $g$ guarantees that $\sigma^{-1}(t,x)\in\mathcal{L}(H,H)$ and $\sigma^{-1}(t,x)H\subset\widetilde{H}_0(t,x)$ for each $t>0,x\in H$.

We clarify how the subspace $\widetilde{H}_0(t,x)$ is chosen: in the current vector field case, if $g$ is uniformly bounded in all its arguments, then we may as well choose $\widetilde{H}_0(t,x)=H$ because in this case multiplication by $g(t,x,u(t,x))$ maps $L^2([0,1];\mathbb{R})$ to $L^2([0,1];\mathbb{R})$. However, if $g$ is not a bounded function in $u$, then multiplication by $g$ is not well-defined from $L^2([0,1];\mathbb{R})$ to $L^2([0,1];\mathbb{R})$. In this case we need to choose $$\widetilde{H}_0(t,x):=\{\frac{1}{g(t,z,x(z))}v(z), v\in L^2([0,1];\mathbb{R}),z\in[0,1]\}.$$ Since $|g|\geq\epsilon$, $\widetilde{H}_0(t,x)\subset H$ with the induced norm. To see that $H_0\subset \widetilde{H}_0(t,x)$, we simply observe that $g(t,\cdot,x(\cdot))H_0\subset H$. 
This justifies assumption $\mathbf{H}_4$.

To justify $\mathbf{H}_5$, assuming that $B$ and $g$ have linear growth in its third argument uniformly in $t$ and $x$, then the claim for $b$ is immediate from definition of $L^2$ norm. The claim for $\sigma$ can be shown as follows: via this specific choice of $H_0$ and $H$, the $\mathcal{L}(H_0,H)$ norm of $\sigma$ can be identified with the $L^2([0,1];\mathbb{R})$-norm of $g(t,z,u(z))$ for $z\in[0,1]$. Observe that the assumptions $\mathbf{H}_3,\mathbf{H}_4$ do not require $g$ to be bounded: $g$ may have a linear growth in its third argument. Meanwhile, other SPDEs that are not in this random field form can be considered as well. For example, one may assume that $g$ depends only on the $L^2$ norm of $u$. The well-posedness of these SPDEs are new in the literature  if we only assume Hölder dependence.

The main novelty of this paper is the Hölder regularity assumption on the noise coefficient, i.e. treatment of assumptions $\mathbf{H}_3$ and $\mathbf{H}_4$ for Hilbert space valued SPDEs. Note that such results are well-known for finite dimensional SDEs via parabolic PDE techniques or the simple trick of representing the solution as a time-changed Brownian motion, but in infinite-dimensional setting none of these techniques are applicable. For earlier results see \cite{zambotti2000analytic} and \cite{athreya2006infinite}. If one is interested in 1-d vector field SPDEs, then the result of Mytnik and Perkins \cite{mytnik2011pathwise} shows that whenever the diffusion coefficient is $\frac{3}{4}+\epsilon$-Hölder continuous and the drift is Lipschitz, then there is a unique strong solution. Our main result has the same $\frac{3}{4}$ threshold as \cite{mytnik2011pathwise}, yet there are a number of differences: on the one hand we need the diffusion coefficient $\sigma$ to be non-degenerate where \cite{mytnik2011pathwise} does not, and we only construct probabilistic weak solutions rather than strong, but on the other hand our construction is valid for vector field SPDEs in any dimension without supposing the coefficient is in a separable form, and far more general SPDEs can be covered by our assumptions beyond vector field case, yet \cite{mytnik2011pathwise} is essentially one-dimensional. We believe that the two solution theories are complementary to each other and has their own strength. In particular, we will derive properties of long term behaviors of solutions and also solve SPDEs with a Burgers type non-linearity as well as non-Lipschitz diffusion coefficient. All these results are to be compared to the well-known fact that for a one dimensional SDE, strong uniqueness is guaranteed as long as the diffusion coefficient is at least $1/2$-Hölder continuous. For example, a result of Zvonkin shows that the 1-d SDE $dX_t=b(t,X_t)dt+\sigma(t,X_t)dB_t,X_0=x_0$ where $b$ is bounded measurable, $\sigma$ is bounded continuous with $|\sigma(t,x)-\sigma(t,y)|\leq\sqrt{|x-y|}$ and $|\sigma|\geq c$, has a strong solution. Other similar results may be found in the book \cite{cherny2005singular}.

After the first version of this paper was posted online, the author further generalized the technique of this paper to second order systems: the stochastic wave equation with non-degenerate Hölder diffusion coefficient \cite{HAN2024110224}. The author investigated the well-posedness issue, and proved the convergence of a damped stochastic wave equation to the stochastic heat equation introduced in this paper as the damping parameter tends to $0$.

\subsection{Well-posedness with Burgers type non-linearity}

Now we show that we can go further to the case where we have a drift of Burgers type $(-A)^{1/2}F(t,X_t)$.

\begin{theorem}\label{burgerstheorem1}
    Consider the following evolution equation where $\vartheta\in(0,1)$
\begin{equation}\label{burgersequation}
dX_t=AX_t dt+b(t,X_t)dt+(-A)^{\vartheta}F(t,X_t)dt+\sigma(t,X_t)dW_t,\quad X_0=x\in H\end{equation}
    on a Hilbert space $H$, where $W$ is a space-time white noise on $H$. 
    
    Assume that the operator $A$ satisfies $\mathbf{H}_1$, the drift $b$ satisfies $\mathbf{H}_2$, the diffusion coefficient $\sigma$ satisfies $\mathbf{H}_3$, $\mathbf{H}_4$ and $\mathbf{H}_5$. Moreover, assume that

    $\mathbf{H}_6$. The map $F(t,x):[0,\infty)\times H\to H$ is measurable and satisfies the following condition: for some $\zeta\in(\vartheta,1]$ and $M>0$, 
    \begin{equation}\label{globalglobal}|F(t,x)-F(t,y)|\leq M|x-y|^\zeta,\quad |x-y|\leq 1,\quad t>0,\end{equation} and that \begin{equation}|F(t,x)|\leq M(1+|x|),\quad x\in H,\quad t>0.\end{equation}

    Then there exists a unique weak mild solution to \eqref{burgersequation}. 
\end{theorem}

\begin{remark}
    In the additive noise case $\sigma=\operatorname{I}_d$ and $b=0$, existence of a unique weak mild solution was proven in \cite{priola2021optimal} in the critical case $\vartheta=\frac{1}{2}$ for all $\zeta\in(0,1]$ via an optimal regularity result for the Kolmogorov equation. Our method can deal with diffusion coefficients that are merely Hölder continuous, yet we restrict ourselves to the more regular regime $\zeta\in(\frac{1}{2},1].$

The method of \cite{priola2021optimal} however does not go beyond $\vartheta=\frac{1}{2}$, and the whole range $\vartheta\in(\frac{1}{2},1)$ is left open. see for example \cite{priola2022correction}, Remark 4. To our best knowledge, the result of Theorem \ref{burgerstheorem1} is new even for additive noise whenever $\vartheta\in(\frac{1}{2},1)$. In \cite{priola2022correction}, Remark 4 the author predicts that there exists some $\zeta$-Hölder ($\zeta>0$) function $F$ such that weak uniqueness no longer holds. Our theorem 
   \ref{burgerstheorem1} does not contradict such predictions, as we only prove weak uniqueness whenever $\zeta>\vartheta$. The other range $\zeta\in(0,\vartheta)$ is not covered here.
The constraint $\zeta\geq\vartheta$ is most likely a consequence of the method we use in this paper, and in particular the constraint arises from the $\lambda^{\vartheta-1}$ factor in the estimate \eqref{burgersusefulestimate}.
   
\end{remark}

We also note that \cite{priola2021optimal} made the assumption that $b$ and $F$ are locally Hölder continuous functions, which means they are Hölder continuous on each ball, whereas we make a global Hölder continuity assumption \eqref{globalglobal}. While this may set a limit to the generality of our results, we believe that our techniques can also be generalized to locally Hölder continuous coefficients via a proper truncation argument. We do not pursue this generalization so as to keep this paper at a reasonable length.

Typical examples of SPDEs that fall within \eqref{burgersequation} are Burgers type equations \footnote{Let $A$ be the Laplacian on $(0,2\pi).$ Setting $F=\frac{\partial}{\partial\xi}(-A)^{-1/2}h$ in \eqref{burgersequation} transforms \eqref{burgersequation} to this equation, and a direct computation shows that $\frac{\partial}{\partial\xi}(-A)^{-1/2}$ is a bounded operator on $L^2((0,2\pi))$. See \cite{prato2003new}, Remark 4.3 and Example 4.4 for related computations.}
\begin{equation}\label{exbur}du(t,\xi)=\frac{\partial^2}{\partial\xi^2}u(t,\xi)dt+\frac{\partial}{\partial \xi} h(u(t,\xi))dt+\sigma(u(t,\xi))dW_t(\xi),u(0,\xi)=u_0(\xi),\quad \xi\in(0,2\pi),\end{equation}
and the Cahn-Hilliard equations with stochastic forcing
\begin{equation}\label{exlug}du(t,\xi)=-\Delta_\xi^2 u(t,\xi)dt+\Delta_\xi h(u(t,\xi))dt+\sigma(u(t,\xi))dW_t(\xi),u(0,\xi)=u_0(\xi) \text{ on } G,\end{equation}
given some open, regular bounded set $G\subset\mathbb{R}^d$, $d=1,2,3$.

\subsection{Long-time behaviour of solutions}

Now we investigate the long time behavior of the solutions we just constructed. 

Assuming that the coefficients $b$, $\sigma$ and $F$ do not depend on time. Then it is easy to check that the solutions to the SPDE \eqref{evolutionequation} and \eqref{burgersequation} define a time homogeneous Markov process on $H$ and satisfy the Feller property. This can be seen from approximating the SPDE by another SPDE with Lipschitz coefficients, and the Markov and Feller property follow from taking the approximation limit.

Assume further the existence of a Lyapunov function, we can show existence and uniqueness of the invariant measure and the existence of a spectral gap.

\begin{theorem}\label{harristheorem}
Assume the assumptions of Theorem \ref{theorem1.1} are satisfied, and coefficients $b$, $\sigma$ do not depend on time. Denote by $\mathbb{E}_x$ the law of the solution $(X_t)_{t\geq 0}$ to \eqref{evolutionequation} with $X_0=x$, and denote by $(\mathcal{P}_t)_{t\geq 0}$ the Markov semigroup generated by $X_t$ on the Hilbert space $H$.

Assume that the following Lyapunov condition holds: for some $t_0>0$, 
\begin{equation}\label{lyapunov1}\mathbb{E}_x V(X_{t_0})-V(x)\leq -cV(x)+C_V,\quad x\in H,\end{equation}
here $V:H\to[1,+\infty)$ is a continuous Lyapunov function, the constants $c\in(0,1)$ and $C_V\in\mathbb{R}$ are fixed and independent of $x$. Assume further that $V(x)\to\infty$ as $|x|\to\infty$.

Then there exists a unique invariant measure $\pi$ for the Markov process $(X_t)_{t\geq 0}$. Moreover, we have the following exponential convergence result:

For any positive $\gamma\in(0,1]$ we set $$d_\gamma(x,y):=|x-y|^\gamma\wedge 1.$$  Meanwhile, for any two Borel probability measures $\mu,\nu$ on $H$, we also write $d_\gamma(\mu,\nu)$ the Wasserstein distance between the measures $\mu$, $\nu$ with respect to the distance $d_\gamma$. Then for any $\gamma\in(0,1]$ and $\epsilon>0$ there exists $\xi>0,C>0$ such that 
    \begin{equation}\label{ratefinal}
d_\gamma(\operatorname{Law}(X_t^x),\pi)\leq Ce^{-\xi t}V^{1-\epsilon}(x),\quad t>0,
    \end{equation}
where $X_t^x$ denotes the solution $X_t$ with initial data $x\in H$.
\end{theorem}

This result generalizes as well for Burgers type SPDEs. We prove the following:

\begin{theorem}\label{burgerstheorem2}
    Assume that all the assumptions of Theorem \ref{burgerstheorem1} are satisfied, and that all the coefficients $b,\sigma,F$ do not depend on time. Then the conclusion of Theorem \ref{harristheorem} is true for $(X_t^x)$, the solution to the Burgers SPDE \eqref{burgersequation}. That is, given a suitable Lyapunov function satisfying \eqref{lyapunov1}, we can deduce the exponential convergence rate \eqref{ratefinal}.
\end{theorem}

\begin{remark}\label{remark111} We give examples for the claimed Lyapunov condition \eqref{lyapunov1}.

Consider the most typical example
\begin{equation}\label{*****}dX_t=A X_t dt+b(X_t)dt+\sigma(X_t)dW_t,\end{equation} where 
$b$ and $\sigma$ are both bounded, and $A$, $b$, $\sigma$ satisfy assumptions $\mathbf{H}_1$ to $\mathbf{H}_5$. The existence of a Lyapunov functional can be verified by taking $V(x):=|x|+1,x\in H$. See Appendix \ref{appendixC} for a proof.

For Burgers type SPDEs, one may also consider
$$dX_t=A X_t dt +(-A)^{\vartheta}F(X_t)dt+\sigma(X_t)dW_t,$$ where $F$ and $\sigma$ are bounded, and $A$, $F$, $\sigma$ satisfy assumptions $\mathbf{H}_1$ to $\mathbf{H}_6$. Then $V(x):=|x|+1$ can again be chosen as a Lyapunov function. The proof is deferred to Appendix \ref{appendixC}.\end{remark}

\subsection{A Banach space example}
\label{section240}
The aforementioned setting, i.e. SPDEs on an abstract Hilbert space, excludes some important SPDE examples such as the stochastic reaction-diffusion equation with an interaction term of polynomial growth. The reaction term cannot be regarded as a map from $L^2(0,1)$ to itself. In this subsection we extend our framework to such an example. Consider the following $\mathbb{R}^r$-valued SPDE
\begin{equation}\label{lastexample}
    \begin{cases}
\frac{\partial u_i}{\partial t}(t,\xi)=\frac{\partial^2u_i}{\partial x_i^2}(t,\xi)+f_i(t,\xi,\vec{u}(t,\xi))+g_i(t,\xi,\vec{u}(t,\xi))\frac{\partial}{\partial t}W_i(t,\xi),t\geq 0,\xi\in(0,1),\\
u_i(0,\xi)=x_i(\xi),\xi\in (0,1),\quad  u_i(t,\xi)=0,\quad t>0,\xi\in\{0,1\}
    \end{cases}
\end{equation}
where $W_i(t,\xi),i=1,\cdots,r$ are 1-d space-time white noises on $L^2([0,1];\mathbb{R})$ and $x_i(\xi)\in \mathcal{C}([0,1];\mathbb{R})$ for each $i$. We use the notation $\vec{u}(t,\xi)$ to denote the vector $$\vec{u}(t,\xi)=(u_1(t,\xi),u_2(t,\xi),\cdots,u_r(t,\xi)).$$

We use the following notation to put this SPDE in our previous framework: we set $H=L^2([0,1];\mathbb{R}^r)$, set $H_0=\mathcal{C}([0,1];\mathbb{R}^r)$, set $A=(\frac{\partial^2}{\partial x_1^2},\cdots,\frac{\partial^2}{\partial x_r^2})$ the vector of Laplacian, set $\mathcal{S}$ the semigroup generated by $A$, set $X(t)\in H$ to be the solution $\vec{u}(t)$, set $F(t,X_t)$ to be the drift vector $(f_i(t,\xi,\vec{u}(t,\xi)))_{i=1}^r$, and set $G(t,X_t)$ to be the diagonal matrix with diagonal terms $g_i(t,\xi,\vec{u}(t,\xi))$ and regard this matrix as an element of $\mathcal{L}(H_0,H).$ Then a solution to this SPDE can be written in the mild form
\begin{equation}
    X_t=\mathcal{S}(t)x+\int_0^t \mathcal{S}(t-s)F(s,X_s)ds+\int_0^t \mathcal{S}(t-s)G(s,X_s)dW_s,
\end{equation}
where $x\in H_0$ denotes the initial condition. Then we prove the following uniqueness and continuity result.
 Note that the functions $f_i,g_i$ are real-valued, so in the following statement the symbol $|f_i(t,\xi,\sigma)|,|g_i(t,\xi,\sigma)|$ simply takes the absolute value of real numbers $f$ and $g$. Also, for $\mathbb{R}^r$-valued vectors $\sigma$, we use the notation $|\sigma|$ to represent its Euclidean norm in $\mathbb{R}^r$. Thus what the symbol $|\cdot|$ means will be very clear from the context.
\begin{theorem}\label{theorem1.8}
    Assume the coefficients $f_i:[0,\infty)\times[0,1]\times\mathbb{R}^r\to\mathbb{R},$ $g_i:[0,\infty)\times[0,1]\times\mathbb{R}^r\to\mathbb{R}$ satisfy the following assumptions:
\begin{enumerate}
    \item(Non-degeneracy and boundedness) There exists $0<C_1<C_2$ such that for any $t>0,\xi\in(0,1)$, $\sigma\in\mathbb{R}^r$ and $i=1,\cdots,r$ we have
    \begin{equation}\label{twosidedboundsg2}
        C_1\leq |g_i(t,\xi,\sigma)|\leq C_2.
    \end{equation}

    \item (Hölder regularity for $g_i$) There exists some $\beta>\frac{3}{4}$ and $M>0$ such that
    \begin{equation}
         \sup_{t>0,\xi\in(0,1),i\in[r]}|g_i(t,\xi,\sigma)-g_i(t,\xi,\sigma')|\leq M|\sigma-\sigma'|^\beta,\quad \sigma,\sigma'\in\mathbb{R}^r,
    \end{equation}
    \item (Separating the drift) The drift $f_i(t,\xi,\sigma)$ can be decomposed into two terms: $$f_i(t,\xi,\sigma)=h_i(t,\xi,\sigma)+k_i(t,\xi,\sigma[i]),\quad \sigma\in\mathbb{R}^r,$$ where $\sigma[i]$ is the $i$-th component of $\sigma$,
\item (Local Hölder regularity for $h_i$, $k_i$) We can find some $\alpha\in(0,1)$ such that, for each $n\in\mathbb{N}_+$ we can find $M_n>0$ such that 
   \begin{equation}
       \sup_{t>0,\xi\in(0,1),i\in[r]}|h_i(t,\xi,\sigma)-h_i(t,\xi,\sigma')|\leq M_n|\sigma-\sigma'|^\alpha, 
   \end{equation}
    \begin{equation}
       \sup_{t>0,\xi\in(0,1),i\in[r]}|k_i(t,\xi,\sigma[i])-k_i(t,\xi,\sigma'[i])|\leq M_n|\sigma-\sigma'|^\alpha, 
   \end{equation}
for any $\sigma,\sigma'\in\mathbb{R}^r$ with $|\sigma|\leq n,|\sigma'|\leq n.$
\item (Linear growth for $h_i$) We assume that there exists some finite $C_3>0$ such that
\begin{equation}\label{linearhi}
    \sup_{t>0,\xi\in(0,1),i\in[i]}|h_i(t,\xi,\sigma)|\leq C_3(1+|\sigma|),\quad \sigma\in\mathbb{R}^r.
\end{equation}

 \item (Polynomial growth for $k_i$) For some $C_4>0$ and $m\geq 1$ we have
    \begin{equation}\label{polyki}
        \sup_{t>0,\xi\in(0,1),i\in[r]}|k_i(t,\xi,\sigma[i])|\leq C_4(1+|\sigma|^m),\quad \sigma\in\mathbb{R}^r
    \end{equation}
    
   \item (Dissipative condition for $k_i$) For some finite $C_5>0$ we have for each $i=1,\cdots,r$, any $\xi\in(0,1)$ and any $\sigma,\rho\in\mathbb{R}^r,$
\begin{equation}\label{dissipative336}
    k_i(t,\xi,\sigma[i])-k_i(t,\xi,\rho[i])=\lambda_i(t,\xi,\sigma,\rho)(\sigma[i]-\rho[i])
\end{equation}
where 
$$
\sup_{t,\xi,\sigma,\rho,i}\lambda_i(t,\xi,\sigma,\rho)<C_5<\infty.
$$
    \end{enumerate}
    Then there exists a unique weak-mild solution to the SPDE \eqref{lastexample}. (It will be proved that almost surely the solution $(\xi\to \vec{u}(t,\xi),\xi\in(0,1))\in H_0$, and thus $(f_i(t,\xi,\vec{u}(t,\xi)))_{1\leq i\leq r}$ is also well-defined and, as a function of $\xi$, lies in $H_0\subset H.$)  Moreover, we prove the following continuity result of solutions with respect to the norm $H_0$: we can define a Wasserstein distance $d$ on the Borel probability space of $H_0$ such that, for any initial condition $x,y\in H_0$ denote by $X_t^x,X_t^y$ the solution to the SPDE with initial value $x,y$ respectively. Then as $x\to y$ in $H_0$, the law of solution $X_t^x$ converges to $X_t^y$ with respect to the Wasserstein distance d:
    \begin{equation}
        d(\operatorname{Law}(X_t^x),\operatorname{Law}(X_t^y))\to 0,\quad |x-y|_{H_0}\to 0.
    \end{equation} Note that $d$ is defined with respect to a distance on $H_0$ rather than $H$.
\end{theorem}

For example, the dissipative condition (7) is satisfied by $k_i(t,\xi,x)=-x^{2n+1}$ for any $n\in\mathbb{N}_+$ but is not satisfied by $k_i(t,\xi,x)=x-x^2$. That is, $k_i$ always tries to drive the dynamics towards zero except for a linear growth part.
For example when $\sigma[i]$ is positive and large, then $k_i(t,\xi,\sigma[i])$ can increase at most linearly in $\sigma[i]$ (if it increases in $\sigma[i]$ faster than any linear function, then in \eqref{dissipative336}, the function $\lambda_i$ cannot be bounded from above) but can decrease polynomially fast in $\sigma[i]$. This dynamics ensure the solution will not blow up in finite time. The long-term behavior of this SPDE can be analyzed via techniques in this paper, but we prefer to leave it for a further investigation to keep this paper short. 

We also note that in \eqref{twosidedboundsg2} we assumed a two-sided bound on the diffusion coefficient $g_i$ and excluded coefficients of linear growth. This restriction arises from the proof technique, originating from \cite{cerrai2003stochastic}. See \cite{cerrai2003stochastic}, Theorem 5.3 and (5.7) therein for a more precise condition.
\subsection{Plan of the paper}
The paper is organized as follows. In Section \ref{section2!} we prove all the well-posedness results for Hilbert space valued SPDEs. In Section \ref{section3!} we prove results concerning long-time behavior of solutions. In Section \ref{section4} we consider the reaction-diffusion equation with polynomial non-linearity. In the appendix we prove some technical lemmas. The proof in Section \ref{section2!} is self-contained, yet for the proof in Section \ref{section3!} we will apply some ergodic theoretical results from \cite{hairer2011asymptotic}. In section \ref{section4}, we prove Theorem \ref{theorem1.8}, where some analytical estimates from \cite{cerrai2003stochastic} will be used.

\section{Well-posedness of weak mild solutions}\label{section2!}

We first give a brief outline of the proof, and focus on all the technical difficulties that arise in infinite dimensions.

The method of this paper is inspired by a generalized coupling technique introduced by Kulik and Scheutzow in \cite{kulik2020well}. In the first step, we find Lipschitz continuous functions $b^n,\sigma^n$ approximating $b$ and $\sigma$, and consider the following SPDE 
$$dX^n_t=A X^n_t dt+b^n(t,X^n_t)dt+\sigma^n(t,X^n_t)dW_t,\quad X^n_0=x,$$
which has a strong solution thanks to the Lipschitz continuity of $b^n$ and $\sigma^n$. The aim of proof is to show that the law of $X_t^n$ converges to the law of $X_t$, so that, given the law of $X_t^n$ is uniquely defined, any candidate solution to $X_t$ should have the same law, i.e. weak uniqueness holds.
However, it is difficult to compare $X_t$ with $X_t^n$ directly. We introduce an auxiliary stochastic process on this probability space via solving the following SPDE
$$d\widetilde{X}^n_t=A\widetilde{X}^n_t dt+b^n(t,\widetilde{X}^n_t)dt+\lambda (X_t-\widetilde{X}_t^n)dt+\sigma^n(t,\widetilde{X}^n_t)dW_t,\quad \widetilde{X}_0^n=x,$$
where $\lambda>0$ is a control parameter that we will set to be very large, and also we will only solve the SPDE $d\widetilde{X}_t^n$ up to some stopping time $\tau$, where the stopping time $\tau$ depends on $X_t$ and $X_t^n$. We will use a pathwise argument to control the distance of $X_t$ and $\widetilde{X}_t^n$, and use a Girsanov transform of measure to control the distance of $X_t^n$ and $\widetilde{X}_t^n$. Combining both controls in a skillful manner will prove the weak uniqueness statement.

Compared with the stochastic delay equation setting in \cite{kulik2020well}, a number of new difficulties arise.
\begin{enumerate} 
 \item First, in infinite dimensions we need particular effort to approximate a continuous map by a Lipschitz one, and to guarantee that the solution stays in some compact subset within finite time horizon. This approximation is taken care of in Proposition \ref{seprateagae} and \ref{proposition1.2}.
 
 \item Second, Itô's formula is not readily applicable to these SPDEs with space-time white noise, thus when we derive a contraction estimate we need very special care to use the stochastic factorization lemma efficiently. The results proved in Section \ref{2.1.3} can be regarded as a maximal inequality of Ornstein-Uhlenbeck semigroups in infinite dimensions, and is thus of independent interest.
 \item Third, in infinite dimensions the verification of Lyapunov condition is not as easy as before, again due to lack of Itô's formula, so we only check the most simple cases in Appendix \ref{appendixC}. We also need some care for the term $(-A)^\vartheta$, which is taken care of in Section \ref{sectionburgures}.
\end{enumerate}

\subsection{Preliminaries} In this subsection we list a number of preparatory results that will be used in the proof.

\subsubsection{Approximation by Lipschitz functions}\label{subsection2.112}  In this paper we work abstractly with mappings on a Hilbert space that are not Lipschitz, and we will frequently need to approximate such mappings by Lipschitz mappings. The following Proposition from Alex Ravsky \cite{610367} would be very useful:

\begin{proposition}\label{seprateagae}
    Consider a separable Hilbert space $V$, a normed space $V'$ and a continuous mapping $f:V\to V'$.  Then there exists a sequence $\{f_n\}$ that are bounded Lipschitz maps $f_n:V\to V'$ such that $f_n$ converges uniformly to $f$ on each compact subset of $V$.
\end{proposition}

We give a very brief sketch of the construction, and refer the details to the reference \cite{610367}. We may simply consider $V=\ell^2$. Let $p_n:\ell^2\to\mathbb{R}^n$ be an orthogonal projection. We define the mapping from $\mathbb{R}^n$ to $V'$ and compose with $p_n$. Then we discretize $\mathbb{R}^n$ into cubes along the coordinate axes, such that each cube has side length $\frac{1}{n}$. We define the approximation map to be equal to $f$ on the vertices of these cubes, and define the value of the approximation function inside the cube via linear interpolation with its values on the cube vertices. The approximating map converges to $f$ on compacts as we send $n$ to infinity.

For any given $t>0$, we may apply this Proposition to $\sigma(t,x):H\to \mathcal{L}(H_0,H)$ where $H$ is separable Hilbert space and $\mathcal{L}(H_0,H)$ is a normed Banach space. By checking the details of construction, one can show that the approximating sequences are also measurable in $t$.  We also apply the construction to $b(t,x):H\to H.$

An important observation is in place: if $f$ has some nice continuity properties, then $f_n$ would also have such properties at least for two points not sufficiently close. Say if $f$ is $\beta$-Hölder continuous with Hölder coefficient $C_\beta$, then $f_n$ satisfies $|f_n(x)-f_n(y)|\leq 2C_\beta |x-y|^\beta$ for any $|x-y|>>\frac{1}{n}$. This is easy to check from the explicit construction procedure given above: since $|x-y|$ is much larger than the side length of the cube used for construction, $f_n(x)-f_n(y)$ is essentially subtracting the values of $f$ on the vertices of the cube containing $x$, and on the vertices of the cube containing $y$. When $n$ is large this is essentially $f(x)-f(y)$ thanks to the continuity of $f$. Meanwhile, when we make the choice $f=\sigma(t,x)$ and $\sigma$ satisfies the non-degeneracy condition $\mathbf{H}_4$, then $f_n$ also satisfies the non-degeneracy condition  $\mathbf{H}_4$ when $n$ is large: this is because $f_n$ is a weighted sum of $f$ on $n$ grid points. We also use the fact that if an operator $R$ on $H$ has a small operator norm, then $\operatorname{Id}_H+R$ admits a right inverse on $H$. More precisely, assume that we can find $g$ such that $fg=\operatorname{Id}_H$, then $f_ng=\operatorname{Id}_H+R$ with $|R|_{\mathcal{L}(H,H)}$ arbitrarily small. Then we find some $h$ so that $(\operatorname{Id}_H+R)h=\operatorname{Id}_H$ and $f_n(gh)=\operatorname{Id}_H$, as desired.

Thanks to these constructions, in the following we will assume that $b(t,x)$ and $\sigma(t,x)$ can be approximated by some sequences $b^n, \sigma^n$ that are Lipschitz continuous and also satisfy assumptions $\mathbf{H}_2$, $\mathbf{H}_3$, $\mathbf{H}_4$ with uniform in $n$ numerical constants for any fixed $x,y\in H.$ (That is, the Hölder constant is uniform for the approximating sequence, whereas the Lipschitz constant may explode.)

\subsubsection{Staying in a compact subset}
Given such an approximation lemma, we also need a result that shows, for any $T>0$, with probability tending to one, the solution of the stochastic heat equation lies in a compact subset $K$ throughout $[0,T]$. Such a lemma follows from the compactness of the semigroup $\mathcal{S}(t)$.
 
\begin{proposition}\label{proposition1.2} Fix $T>0$ and consider the stochastic evolution equation \eqref{evolutionequation} satisfying the assumptions of Theorem \ref{theorem1.1}.
For any $\epsilon>0$, we can find a compact subset $K\subset H$ such that 
$\mathbb{P}(\theta_K\leq T)<\epsilon,$ where $\theta_K:=\inf\{t\geq 0:X_t\notin K\}.$ The same is true for the SPDE \eqref{burgersequation} with Burgers type non-linearity.
\end{proposition}
The proof of this Proposition is deferred to Appendix \ref{appendixA} and  \ref{appendixB}.

\subsubsection{The maximal inequality estimate}\label{2.1.3}
We begin with a deviation estimate for stochastic integrals in the white noise case. This result is fundamental  for all the other proofs of this paper.
In the following we will consider an adapted process $\Phi\in L^\infty([0,T];\mathcal{L}(H_0,H))$
and its stochastic integral. For this purpose we introduce the following notation: for any $T>0$,
\begin{equation}
\|\Phi\|_T:=\sup_{0\leq t\leq T}|\Phi(t)|_{\mathcal{L}(H_0,H)}.
\end{equation}
Throughout the paper we will need a maximal inequality estimate for the process
$$\Gamma(t):=\int_0^t S(t-s)e^{-\lambda (t-s)}\Phi(s)dW_s,$$
and we aim to show that $|\Gamma(t)|$ is uniformly small when $\lambda$ is very large. If we were in a finite dimension setting we could use Itô's formula for any such derivations. Unfortunately, Itô's formula is not available now. We use instead the stochastic factorization lemma from \cite{da2014stochastic}, which says: for any $\alpha>0$:
\begin{equation}\label{2.2maim2.2}
    \Gamma(t)=\frac{\sin(\alpha\pi)}{\pi}\int_0^t (t-s)^{\alpha-1}\mathcal{S}(t-s)e^{-\lambda(t-s)}Z(s)ds,
\end{equation}

where
\begin{equation}
    Z(t)=\int_0^t (t-s)^{-\alpha}\mathcal{S}(t-s)e^{-\lambda(t-s)}\Phi(s)dW_s.
\end{equation}

The first step is to estimate $Z(t)$.
\begin{proposition}\label{relabel}
Consider the stochastic integral of an adapted process $\Phi$ 
$$Z(t):=\int_0^t(t-s)^{-\alpha} S(t-s)e^{-\lambda (t-s)}\Phi(s)dW_s,$$
where the semigroup $S(t)$ is generated by an operator $A$ satisfying $\mathbf{H}_1$ and $\lambda>0$ is some fixed constant. Then for each  $\eta\in(0,\eta_0),$ we can choose some $\alpha>0$ small enough (only depending on $\eta$) such that we have the following estimate: for any $p\geq 2$,
\begin{equation}
    \mathbb{E}[|Z(t)|^p]\leq C_{\eta,p,\alpha,T}\|\Phi\|_t^p \lambda^{-\frac{\eta p}{2}},\quad \text{ for any } t\in[0,T].
\end{equation}
The constant $C_{\eta,p,\alpha,T}<\infty$ depends on $\eta,p$, $\alpha$ and $T$.

\end{proposition}

\begin{proof} 
We write $\mathcal{S}^\lambda$ the semigroup generated by $-A-\lambda I$, so that $$Z(t)=\int_0^t \mathcal{S}^\lambda(t-s)\Phi(s)dW_s.$$ By the BDG inequality,
\begin{equation}
    \mathbb{E}[|Z(t)|_H^p]\leq C\mathbb{E}\left( \sum_{j=1}^\infty\int_0^t (t-s)^{-2\alpha}|\mathcal{S}^\lambda(t-s)\Phi(s)e_j|_H^2ds\right)^{p/2}.
\end{equation}
 We estimate the right hand side via:
$$\begin{aligned}
\Lambda_\alpha^{\lambda}(t):&=\sum_{j=1}^\infty \int_0^t (t-s)^{-2\alpha}|\mathcal{S}^\lambda(t-s) \Phi(s)e_j|_H^2ds\\
&=\sum_{k=1}^\infty \sum_{j=1}^\infty \int_0^t (t-s)^{-2\alpha}\langle \mathcal{S}^\lambda(t-s)\Phi(s)e_j,e_k\rangle_H^2ds\\
&=\sum_{k=1}^\infty\sum_{j=1}^\infty\int_0^t (t-s)^{-2\alpha}\langle \Phi(s)e_j,S^{\lambda*}(t-s)e_k\rangle_H^2ds
\end{aligned}
$$
 where $\mathcal{S}^{\lambda_*}(t)$ is the adjoint operator of the operator $\mathcal{S}^\lambda(t)$ in the Hilbert space, in the sense that $\langle S^\lambda(t)h_1,h_2\rangle_H=\langle S^{\lambda_*}(t)h_2,h_1\rangle_H$ for any $h_1,h_2\in H$.

Since $$\begin{aligned}
\langle\mathcal{S}^{\lambda*}(t)^*e_k,e_j\rangle_H&=\langle \mathcal{S}^\lambda(t) e_j,e_k\rangle_H\\&=
\begin{cases}e^{-(\lambda_j+\lambda)t}\quad j=k\\0\quad j\neq k,\end{cases}
\end{aligned}$$ we further simplify $\Lambda_\alpha^{\lambda}(t)$ as

$$
\Lambda_\alpha^{\lambda}(t)=\sum_{k=1}^\infty\sum_{j= 1}^\infty\int_0^t (t-s)^{-2\alpha}e^{-2(\lambda_k+\lambda)(t-s)}\langle \Phi(s)e_j,e_k\rangle_H^2ds.
$$
By self-adjointness of $\Phi$, noting $e_k\in H_0$ with $|e_k|_{H_0}\leq M$,
$$\sum_{j=1}^\infty \langle \Phi(s)e_j,e_k\rangle_H^2=\sum_{j=1}^\infty \langle \Phi(s)e_k,e_j\rangle_H^2=|\Phi(s)e_k|_H^2\leq M|\Phi(s)|^2_{\mathcal{L}(H_0,H)}.$$

    Thus we upper $\Lambda_\alpha^{\lambda}(t)$ through
    \begin{equation}\label{Lambdamulambda}
\Lambda_\alpha^{\lambda}(t)\leq {M}\int_0^t (t-s)^{-2\alpha}\|\Phi(s)\|_t^2\left(\sum_{k=1}^\infty e^{-2(\lambda_k+\lambda)(t-s)}
\right)ds.
    \end{equation}

In the following we find effective upper bounds for $$\sum_{k=1}^\infty e^{-2(\lambda_k+\lambda)(t-s)}$$
For any $\eta\in(0,1)$, by Young's inequality we may find $c_\eta>0$ such that \begin{equation}\label{young's}
\lambda_k+\lambda\geq c_\eta {\lambda_k}^{1-\eta}\lambda^\eta.\end{equation} 
For any $r<1$ we can find positive constants $c_r$, $c_{r,\eta}$ such that, by using \eqref{young's} in the second inequality,
\begin{equation}\label{J1J1}e^{-2(\lambda_k+\lambda)t}\leq c_r\frac{1}{((\lambda_k+\lambda)t)^r}\leq c_{r,\eta}\frac{1}{\lambda_k^{(1-\eta)r}}\lambda^{-\eta r}\frac{1}{t^r}. \end{equation}
Since \begin{equation}\label{579}\int_0^t s^{-2\alpha-r} ds<\infty\end{equation} by choosing $\alpha>0$ sufficiently small followed by choosing $r<1$ sufficiently close to 1, one can ensure that $\sum_{k=1}^\infty\frac{1}{\lambda_k^{(1-\eta)r}}<\infty$ for $r$ sufficiently close to 1, and then $r\eta$, the exponent of $\lambda^{-1},$ can achieve any value in $(0,\eta_0)$ by our choice. Thus we write $\eta$ in place of $r\eta$ for the given $\eta\in(0,1)$, and derive the estimate \begin{equation}\label{eqref580}
   \sum_{k=1}^\infty e^{-2(\lambda_k+\lambda)t}\leq c_{\eta} \lambda^{-\eta}\frac{1}{t^r}
\end{equation} where $c_\eta$ implicitly depends on $r$ but is independent of $t$. Note that in the above procedure, $r$ depends only on $\eta$, and thus the admissible choice of $\alpha$ only depends on $\eta$ as well.

This produces the desired $\lambda^{\eta}$ factor in the upper bound of $\Lambda_\alpha^{\lambda}$, and the integral with respect to $t$ is finite thanks to \eqref{579}.
 This finishes the proof with a constant $c_{\eta,p,\alpha,T}$.

\end{proof}

Now we turn back to the original process $\Gamma(t)$.

\begin{theorem}\label{stochasticinte}
    Recall the process $\Gamma(t)$ defined as
     $$\Gamma(t)=\int_0^t \mathcal{S}(t-s)e^{-\lambda(t-s)}\Phi(s)dW_s.$$
     Then for any $\eta\in(0,\eta_0)$ we can find a $p_*>0$ depending only on $\eta$ such that, for any $p\geq p_*$, we can find $c_{\eta,p,T}>0$ such that for any $\lambda>0$, 
     \begin{equation}
         \mathbb{E}\sup_{t\in[0,T]}|\Gamma(t)|^p\leq c_{\eta,p,T}\lambda^{-\frac{\eta p}{2}}\mathbb{E}\|\Phi\|_T^p.
     \end{equation} The constant $c_{\eta,p,T}$ depends on $\eta$, $p$ and $T$.

     More generally, for any stopping time $\tau$ we can estimate
      \begin{equation}
         \mathbb{E}\sup_{t\in[0,T\wedge\tau]}|\Gamma(t)|^p\leq c_{\eta,p,T}\lambda^{-\frac{\eta p}{2}}\mathbb{E}\|\Phi\|_{T\wedge\tau}^p.
     \end{equation}

     \begin{proof} 
         By the stochastic factorization formula \eqref{2.2maim2.2},
   we have, for some $\alpha>0$ (determined in Proposition \ref{relabel} by $\eta$), the first equality of the following:
   \begin{equation}\begin{aligned}
\mathbb{E}&\sup_{t\in[0,T]}|\Gamma(t)|^p=\mathbb{E}\sup_{t\in[0,T]}\left|\int_0^t (t-s)^{\alpha-1}\mathcal{S}^\lambda(t-s)Z(s)ds\right|^p
\\&
\leq C_{\eta,p,T}\left(\int_0^T s^{(\alpha-1)p/(p-1)}\|\mathcal{S}^\lambda(s)\|^{p/(p-1)}ds\right)^{p-1}\mathbb{E}\int_0^T |Z(s)|^p ds
\\&\leq C_{\eta,p,T}\lambda^{-\frac{\eta p}{2}}\mathbb{E}\|\Phi\|_T^p,
\end{aligned}   \end{equation}
     where in the second inequality we choose $p>\frac{1}{\alpha}$ sufficiently large, use the boundedness of $\|\mathcal{S}^\lambda(s)\|$ and apply Hölder's inequality. The last inequality follows from Proposition \ref{relabel}. This completes the proof.

     To prove the stopping time version, we may replace $\Phi(s)$ by $\Phi(s)1_{s\leq\tau}$ in the definition of $\Gamma(t)$ without changing the value of $\Gamma(t)$ for $t\leq T\wedge\tau$, and do the same computation.

     \end{proof}

\end{theorem}

\subsubsection{The Girsanov transform part}
Finally we provide an estimate 
that follows from Girsanov transform, which is an analogue of Proposition 3.2 of \cite{kulik2020well} in the SPDE setting. 

\begin{proposition}\label{proposition2.3}
Consider $X_t$ that solves the SPDE \eqref{evolutionequation} under the assumption of Theorem \ref{theorem1.1}, and consider another SPDE 
\begin{equation}\label{girsanov!}dY_t=AY_t dt+b(t,Y_t)dt+\lambda \phi_t  dt+\sigma(t,Y_t)dW_t,\quad Y_0=x\in H,\end{equation}
where $(\phi_t)_{t\geq 0}$ is an $H$-valued adapted process. We assume that $b$ and $\sigma$ are Lipschitz continuous, so that \eqref{girsanov!} has a unique strong solution. Assume that $\sup_{0\leq t\leq T}|\phi_t|\leq |\phi|_\infty<\infty$ almost surely for some finite constant $|\phi|_\infty.$ Then for any $T>0$ we may find a constant $C$ such that 
$$d_{TV}(\operatorname{Law}(X|_{[0,T]}),\operatorname{Law}(Y_{[0,T]}))\leq CT^{1/2}\lambda |\phi|_\infty,$$
where $d_{TV}$ is the total variation distance on the Borel probability space on $\mathcal{C}([0,T];H)$.
\end{proposition}

\begin{proof}
By Pinsker's inequality, we have
$$d_{TV}(\operatorname{Law}(X|_{[0,T]}),\operatorname{Law}(Y_{[0,T]}))\leq\sqrt{2 H_{\text{rel}}(\operatorname{Law}(X|_{[0,T]})\mid \operatorname{Law}(Y_{[0,T]}))},$$
where $H_{\text{rel}}(\cdot\mid\cdot)$ denotes the relative entropy on path space $\mathcal{C}([0,T];H).$

Since $\sigma$ is non-degenerate, the law of $X|_{[0,T]}$ can be obtained from $Y|_{[0,T]}$ via Girsanov transform (a version of Girsanov transform for cylindrical processes can be bound in \cite{da2014stochastic}, Chapter 10):
$$(W_t)_{0\leq t\leq T}\mapsto (W_t-\sigma(t,Y_t)^{-1}\lambda \phi_t )_{0\leq t\leq T}.$$

Since $\sigma^{-1}\lambda\phi_t$ is bounded almost surely by $C\lambda|\phi|_\infty$, where $C$ is the upper bound in Hypothesis $\mathbf{H}_4$, the claim follows from a direct computation of relative entropy. Namely, let $W_t$ be the white noise, let $\mu_W$ be the distribution of the stochastic heat equation $dX_t=AX_tdt+dW_t,X_0=0$, and let $\xi$ be an SPDE solving $\xi_0=0$ and 
$$
d\xi_t=A\xi_tdt+\beta_t dt+dW_t,
$$ whose distribution is denoted by $\mu_\xi$. Then via computing the Girsanov change of density we get the following formula (see \cite{butkovsky2020generalized}, Appendix A or \cite{da2014stochastic} for Girsanov transform in infinite dimensions)
$$
H_{rel}(\mu_\xi,\mu_W)\leq\frac{1}{2}\mathbb{E}\int_0^\infty|\beta_t^2|dt.
$$ Applying this inequality to our change of measure, we get that $$2H_{\text{rel}}(\operatorname{Law}(X|_{[0,T]})\mid \operatorname{Law}(Y_{[0,T]}))\leq C^2T\lambda^2|\phi|_\infty^2,$$ where $C$ is any constant satisfying $|\sigma^{-1}(t,x)|\leq C\quad\forall t,x,$. Remarkably, this relative entropy bound does not depend on the Lipschitz continuity of the various coefficients but only on the supremum norm. This is the key reason enabling our argument to work.
\end{proof}

\subsection{Proof of Theorem \ref{theorem1.1}}

\begin{proof} Existence of a weak mild solution is well-known, see for example \cite{gkatarek1994weak}, Theorem 2, and a generalized version of the existence result will be proved in Appendix \ref{appendixB}. We now prove uniqueness of weak mild solution.

As discussed in Section \ref{subsection2.112}, we find a sequence of Lipschitz mappings $b^n$, $\sigma^n$ approximating $b$ and $\sigma$ on compacts, and by the discussion at the end of that subsection, we may assume $b^n$, $\sigma^n$ also satisfy assumptions $\mathbf{H}_2,\mathbf{H}_3,\mathbf{H}_4$ with the same numerical constant.

For any initial value $x\in H$, consider the following three SPDEs
$$dX_t=A X_t dt+b(t,X_t)dt+\sigma(t,X_t)dW_t,\quad X_0=x,$$ which is the SPDE we wish to prove uniqueness in law.  As we have not shown the solution is unique in law, we temporarily select one candidate solution and fix it in the sequel. Therefore, we assume the selected weak solution to this SPDE lives on some probability space $(\Omega,\mathcal{F},(\mathcal{F}_t),W).$ On the same probability space where $X_t$ lives, we solve the following SPDE
$$dX^n_t=A X^n_t dt+b^n(t,X^n_t)dt+\sigma^n(t,X^n_t)dW_t,\quad X^n_0=x,$$
which has a strong solution thanks to the Lipschitz continuity of $b^n$ and $\sigma^n$. The aim of proof is to show that the law of $X_t^n$ converges to the law of $X_t$, so that, given the law of $X_t^n$ is uniquely defined, any candidate solution to $X_t$ should have the same law, i.e. weak uniqueness holds.
Also, on this probability space we solve the following SPDE
$$d\widetilde{X}^n_t=A\widetilde{X}^n_t dt+b^n(t,\widetilde{X}^n_t)dt+\lambda (X_t-\widetilde{X}_t^n)dt+\sigma^n(t,\widetilde{X}^n_t)dW_t,\quad \widetilde{X}_0^n=x.$$
This SPDE has a strong solution on this probability space thanks to Lipschitz continuity of $b^n$ and $\sigma^n$.

Before going into details we briefly outline the two-step proof strategy. To show $X_t^n$ converges to $X_t$ in distribution, we first establish a contracting estimate of $X_t$ towards $\widetilde{X}_t^n$ via choosing $\lambda$ sufficiently large. Then we show that $X_t^n$ is sufficiently close to $\widetilde{X}_t^n$ in total variation distance, using an appropriate definition of the stopping time $\tau$ and non-degeneracy of $\sigma^n$. The proof completes as a combination of these two bounds.

Given a compact subset $K\subset H$ with $x\in K$ obtained via Proposition \ref{proposition1.2}, denote by $$\Delta_K^n:=\sup_{t>0}\sup_{y\in K}\{|b^n(t,y)-b(t,y)|+|\sigma^n(t,y)-\sigma(t,y)|_{\mathcal{L}(H_0,H)}\}$$
and set $\tau_K^n:=\inf\{t\geq 0:|X_t-\widetilde{X}_t^n|\geq 2\Delta_K^n\}.$ Also denote by $\theta_K:=\inf\{t\geq 0: X_t\notin K\}.$ We set $$\tau=\tau_K^n\wedge\theta_K$$ and $$\lambda=({\Delta_K^n})^{\gamma -1}$$ 
for some value of $\gamma>0$ to be fixed later. Thanks to Proposition \ref{proposition1.2}, we have $\mathbb{P}(\theta_K\leq T)\leq\epsilon.$ We stress that the stopping time $\tau$ is uniquely defined once the law of $X_t$ is determined.

\textbf{Step 1: Contraction estimate.} For $0\leq t\leq \tau$, assume for simplicity that $\Delta_K^n\in[0,1)$, we may compute via triangle inequality that for some $n$-independent constant $C>0$,
$$
    |b(t,X_t)-b^n(t,\widetilde{X}_t^n)|\leq M(2\Delta_K^n)^{\alpha}+\Delta_K^n\leq C (\Delta_K^n)^\alpha,\quad t\leq \tau
$$ and 
$$|\sigma(t,X_t)-\sigma^n(t,\widetilde{X}_t^n)|_{\mathcal{L}(H_0,H)}\leq M(2\Delta_K^n)^\beta+\Delta_K^n\leq C (\Delta_K^n)^\beta,\quad t\leq\tau,$$
where $M$ is the constant appearing in Assumptions $\mathbf{H}_2$ and $\mathbf{H}_3$.

A direct computation yields that, for any $0\leq t\leq\tau$,
\begin{equation}\label{linearexpansions}\begin{aligned}
|X_t-\widetilde{X}_t^n|&\leq e^{-\lambda t}\left|\int_0^t S(t-s)e^{\lambda s }[b(s,X_s)-b^n(s,\widetilde{X}_s^n)]ds\right|\\&+\left|e^{-\lambda t}\int_0^t S(t-s)e^{\lambda s }[\sigma(s,X_s)-\sigma^n(s,\widetilde{X}_s^n)]dW_s\right|.\end{aligned}.\end{equation}

 Applying the stopping time version of Theorem \ref{stochasticinte}, note that we choose $\lambda=(\Delta_K^n)^{\gamma -1}$, we obtain that for each $\eta\in(0,\eta_0)$ and sufficiently large $m\in\mathbb{N}_+$, we may find a constant $C=C_{\eta,m,T}$ satisfying
\begin{equation}\label{complementary}\mathbb{P}(\sup_{t\in[0,\tau]}|X_t-\widetilde{X}_t^n|>C (\Delta_K^n)^{1+\alpha-\gamma}+C(\Delta_K^n)^{\frac{1-\gamma}{2}\eta+\beta}R)\leq \frac{C}{R^m},\quad R>0,\end{equation}
where the term $(\Delta_K^n)^{\frac{1-\gamma}{2}\eta+\beta}$ is obtained from Theorem \ref{stochasticinte} and the term $(\Delta_K^n)^{1+\alpha-\gamma}$ follows from a more elementary argument applied to the first term on the right hand side of  \eqref{linearexpansions}: $\int_0^te^{-\lambda(t-s)}(\Delta_K^n)^\alpha ds\leq (\Delta_K^n)^{1-\gamma+\alpha}$.

By our assumption $\alpha\in(0,1]$ and $\beta\in(1-\frac{\eta_0}{2},1]$, we can find $\eta\in(0,\eta_0)$ sufficiently close to $\eta_0$ and $\gamma\in(0,1)$ sufficiently close to $0$, such that the following holds at the same time:
\begin{equation}
    \label{condition2}
\begin{cases}
\gamma\in(0,\alpha)\\
\frac{1-\gamma}{2}\eta+\beta>1.\\
\end{cases}\end{equation}
Consequently, the powers of $\Delta_K^n$ in the bracket of \eqref{complementary} are all larger than one.
We fix from now on the choice of $\gamma$ and $\eta$, as well as the constant $C_m=C_{\eta,m}$.

Now we choose $\chi\in(0,\frac{1-\gamma}{2}\eta+\beta-1]$ and set $R:=(\Delta_K^n)^{-\chi}$ in the estimate \eqref{complementary}. Denote by $$\Omega_\nu:=\{\sup_{t\in[0,\tau]}|X_t-\widetilde{X}_t^n|>C (\Delta_K^n)^{1+\alpha-\gamma}+C(\Delta_K^n)^{\frac{1-\gamma}{2}\eta+\beta-\chi}\}.$$
Upon choosing $\Delta_K^n$ sufficiently small, we deduce that $|X_t-\widetilde{X}_t^n|\leq\Delta_K^n$ for all $t\leq\tau$ on $\Omega\setminus\Omega_\nu$. Since the processes have continuous trajectories, we deduce that we indeed have $$\theta_K\wedge T\leq \tau_K^n$$ on $\Omega\setminus\Omega_\nu$. From this we conclude that, for any $\kappa>0$,
\begin{equation}
   \label{convergeinprob} 
\mathbb{P}(\sup_{t\in[0,T\wedge \theta_K]} |X_t-\widetilde{X}_t^n|>\kappa)\to 0,\quad n\to\infty.\end{equation}

The previous computations also imply that $\mathbb{P}(\Omega_\nu)\to 0$, so that
\begin{equation}\label{whatimplythat}
    \mathbb{P}(\tau_K^n< \theta_K\wedge T)\to 0,\quad n\to\infty.
\end{equation}

\textbf{Step 2: Girsanov transform estimate.} We define one more auxiliary process for convenience: consider the following SPDE
$$d\widetilde{\underline{X}}^n_t=A\widetilde{\underline{X}}^n_t dt+b^n(t,\widetilde{\underline{X}}^n_t)dt+\lambda (X_t-\widetilde{\underline{X}}_t^n)1_{t
\leq\tau}dt+\sigma^n(t,\widetilde{\underline{X}}^n_t)dW_t,\quad \widetilde{\underline{X}}_0^n=x.$$
Since we have determined the process $X_t$ and the stopping time $\tau_K^n$ is uniquely defined, and that the coefficients $b^n,\sigma^n$ are Lipschitz, the SPDE $\widetilde{\underline{X}}^n_t$ has a unique strong solution on this probability space. Moreover we have $\widetilde{\underline{X}}^n_t=\widetilde{X}^n_t$ whenever $t\leq\tau$.

Now we apply Proposition \ref{proposition2.3} for the process $X_t^n$ perturbed by $\lambda(X_t-\widetilde{X}_t^n)1_{t\leq\tau}$ (the SPDE solved by  $\widetilde{X}_t^n$ has Lipschitz coefficients, fulfilling the requirement of Proposition \ref{proposition2.3}) and deduce that, 
for any $T>0$ we can find a constant $C$ (uniform in $n$) such that
\begin{equation}\label{variationineq}
d_{TV}(\operatorname{Law}(X^n|_{[0,T]}),\operatorname{Law}(\widetilde{\underline{X}}^n|_{[0,T]}))\leq C T^{1/2}(\Delta_K^n)^\gamma.\end{equation}

We use the fact that $\lambda|X_t-\widetilde{\underline{X}}_t^n|=\lambda|X_t-\widetilde{X}_t^n|\leq 2\lambda \Delta_K^n$ whenever $t\leq\tau$.

\textbf{Step 3: Combining all the above estimates.}
To prove uniqueness in law of the solution $X_t$, it suffices to show that, for any bounded continuous function $E$ on $\mathcal{C}([0,T];H)$, we have 
\begin{equation}\label{condition}
\mathbb{E}[E(X|_{[0,T]})]-\mathbb{E}[E(X^n|_{[0,T]})]\to 0,\quad n\to\infty,\end{equation}
Indeed, since $b^n$ and $\sigma^n$ are Lipschitz for each $n$, the SPDE solved by $X^n$ has a unique strong solution, and $\operatorname{Law}(X^n|_{[0,T]})$ is uniquely determined. If \eqref{condition} were established, then it would imply that
$$
\mathbb{E}[E(X|_{[0,T]})]$$
is uniquely defined for any continuous test function $E$ and any candidate solution $X$, so the law of $X|_{[0,T]}$ is unique. Since $T$ is arbitrary, this shows the uniqueness in law of $X$.

To prove \eqref{condition}, first note that \eqref{variationineq} implies,  in the limit $\Delta_K^n\to 0,$
$$\mathbb{E}[E(\widetilde{\underline{X}}^n|_{[0,T]})]-\mathbb{E}[E(X^n|_{[0,T]})]\to 0,\quad n\to\infty.$$
Moreover, since $\mathbb{P}(\theta_K\leq T)$ can be taken arbitrarily small thanks to Proposition \ref{proposition1.2} and we have \eqref{whatimplythat}, we deduce that 
$$|\mathbb{E}[E(\widetilde{\underline{X}}^n|_{[0,T]})]-\mathbb{E}[E(\widetilde{X}^n|_{[0,T]})]|\leq 2\sup_x |E(x)|\times\mathbb{P}(\tau< T),$$ which can be made arbitrarily small by setting $K$ large and setting $n\to\infty$, using \eqref{whatimplythat}.

Note also that \eqref{convergeinprob} implies $\widetilde{X}^n|_{[0,T]}$ converges to $X|_{[0,T]}$ in probability, on $\{\theta_K\geq T\}$. Therefore 
$$\lim\sup_{n\to\infty}\left|\mathbb{E}[E(\widetilde{X}^n|_{[0,T]})]-\mathbb{E}[E(X|_{[0,T]})]\right|\leq  2\sup_x|E(x)|\mathbb{P}(\theta_K\leq T),\quad n\to\infty.$$ Combining the aforementioned three estimates, we have
$$\lim\sup_{n\to\infty}\left|\mathbb{E}[E(X^n|_{[0,T]})]-\mathbb{E}[E(X|_{[0,T]})]\right|\leq  4\sup_x|E(x)|\mathbb{P}(\theta_K\leq T).$$
The right hand side can be set arbitrarily small thanks to Proposition \ref{proposition1.2}. This completes the proof of \eqref{condition}, and weak uniqueness follows.

\end{proof}

\subsection{Proof of Theorem \ref{burgerstheorem1}}\label{sectionburgures}

In this section we prove Theorem \ref{burgerstheorem1}, where there is an additional $(-A)^{\vartheta}F(t,X_t)$ term.

We need a simple estimate:

\begin{equation}\label{05.1} \|(-A)^{\vartheta}S(t)\|_{op}=\sup_{k\geq 1}e^{-\lambda_k t} \lambda_k^{\vartheta}\leq \frac{d_\vartheta}{t^{\vartheta}},\end{equation}
where $$d_\vartheta:=\sup_{r> 0}e^{-r}r^{\vartheta}<\infty.$$

Moreover, for any $\lambda>0$ we have
\begin{equation}\label{burgersusefulestimate}
\begin{aligned}
    \int_0^t e^{-\lambda(t-s)}\|(-A)^{\vartheta}S(t-s)\|_{op}ds&\leq d_\vartheta\int_0^t e^{-\lambda s}\frac{1}{s^\vartheta}ds
    \\&\leq d_\vartheta \lambda^{\vartheta-1} \int_0^\infty e^{-t}\frac{dt}{t^\vartheta}\leq C_\vartheta\lambda^{\vartheta-1}
    \end{aligned}
\end{equation}
for some universal constant $C_\vartheta>0$ that only depends on $\vartheta\in(0,1)$.

 We now begin to prove the well-posedness result, Theorem \ref{burgerstheorem1}.

\begin{proof}
For existence of a weak mild solution, we generalize the existence part of Theorem 1 of \cite{priola2022correction} to cover multiplicative noise. Proof of weak existence is given in Appendix \ref{appendixB}.

We now focus on showing the weak solution is unique in law. We follow closely the proof of Theorem \ref{theorem1.1}. Consider approximations $b^n$, $\sigma^n$ and $F^n$ that are Lipschitz continuous and satisfy $\mathbf{H}_2$, $\mathbf{H}_3$, $\mathbf{H}_5$ and $\mathbf{H}_6$ with constants that are uniform in $n$, and such that $b^n,\sigma^n$ and $F^n$ converges uniformly to $b,\sigma$ and $F$ on any compact subset of $H$. We select a weak solution to the following SPDE that lives on some probability space $(\Omega,\mathbb{P},\mathcal{F})$,
$$dX_t=A X_t dt+b(t,X_t)dt+(-A)^{\vartheta}F(t,X_t)dt+\sigma(t,X_t)dW_t,\quad X_0=x,$$ and on the same probability space we solve the following two SPDEs thanks to Lipschitz continuity of various coefficients, given the adapted process $X_t$:
$$dX^n_t=A X^n_t dt+b^n(t,X^n_t)dt+(-A)^{\vartheta}F^n(t,X^n_t)dt+\sigma^n(t,X^n_t)dW_t,\quad X^n_0=x,$$
 and
$$\begin{aligned}
d\widetilde{X}^n_t&=A \widetilde{X}^n_t dt+b^n(t,\widetilde{X}^n_t)dt\\&+(-A)^\vartheta F^n(t,\widetilde{X}^n_t)dt+\lambda (X_t-\widetilde{X}_t^n)dt +\sigma^n(t,\widetilde{X}^n_t)dW_t,\quad \widetilde{X}_0^n=x.\end{aligned}$$

Given a compact subset $K\subset H$ with $x\in K$ constructed via Proposition \ref{proposition1.2}, denote by $$\Delta_K^n:=\sup_{t>0}\sup_{y\in K}\{|b^n(t,y)-b(t,y)|+|\sigma^n(t,y)-\sigma(t,y)|_{\mathcal{L}(H_0,H)}+|F^n(t,y)-F(t,y)|\},$$
and set $\tau_K^n:=\inf\{t\geq 0:|X_t-\widetilde{X}_t^n|\geq 2\Delta_K^n\}.$ Also denote by $\theta_K:=\inf\{t\geq 0: X_t\notin K\}.$ We set $$\tau=\tau_K^n\wedge\theta_K$$ and $$\lambda=({\Delta_K^n})^{\gamma -1}$$ 
for some value of $\gamma>0$ to be fixed later.

\textbf{Step 1. Contraction estimate.}
For $0\leq t\leq \tau$, assume for simplicity that $\Delta_K^n\in[0,1)$, we may compute via triangle inequality that for some $n$-independent constant $C>0$,
$$
    |b(t,X_t)-b^n(t,\widetilde{X}_t^n)|\leq M(2\Delta_K^n)^{\alpha}+\Delta_K^n\leq C (\Delta_K^n)^\alpha,\quad t\leq \tau
$$ 
$$|\sigma(t,X_t)-\sigma^n(t,\widetilde{X}_t^n)|_{\mathcal{L}(H_0,H)}\leq M(2\Delta_K^n)^\beta+\Delta_K^n\leq C (\Delta_K^n)^\beta,\quad t\leq\tau,$$
and 
 $$|F(t,X_t)-F^n(t,\widetilde{X}_t^n)|\leq M(2\Delta_K^n)^{\zeta}+\Delta_K^n\leq C (\Delta_K^n)^\zeta,\quad t\leq \tau,$$
where $M$ is the constant appearing in Assumptions $\mathbf{H}_2$, $\mathbf{H}_3$ and $\mathbf{H}_6$.

A direct computation yields that, for any $t\leq\tau$,
$$\begin{aligned}
|X_t-\widetilde{X}_t^n|&\leq e^{-\lambda t}\left|\int_0^t S(t-s)e^{\lambda s }[b(s,X_s)-b^n(s,\widetilde{X}_s^n)]ds\right|\\&+\left|e^{-\lambda t}\int_0^t S(t-s)e^{\lambda s }[\sigma(s,X_s)-\sigma^n(s,\widetilde{X}_s^n)]dW_s\right|\\&+\left|\int_0^t e^{-\lambda(t-s)}(-A)^{\vartheta}S(t-s)[F(s,X_s)-F^n(s,\widetilde{X}_s^n)]ds\right|
.\end{aligned}$$

Now arguing as in \eqref{complementary}, using the auxiliary estimate \eqref{burgersusefulestimate}, note that we choose $\lambda:=(\Delta_K^n)^{\gamma -1}$, we deduce that 
 for each $\eta\in(0,\eta_0)$ and each sufficiently large  $m\in\mathbb{N}_+$, we may find a constant $C=C_{\eta,m}$ satisfying, for any $R>0$,
\begin{equation}\label{burgerscomplementary}\mathbb{P}(\sup_{t\in[0,\tau]}|X_t-\widetilde{X}_t^n|>C (\Delta_K^n)^{1+\alpha-\gamma}+C (\Delta_K^n)^{(1-\gamma)(1-\vartheta)+\zeta}+C(\Delta_K^n)^{\frac{1-\gamma}{2}\eta+\beta}R)\leq \frac{C}{R^m},\quad R\geq 2.\end{equation}

By our assumption $\alpha\in(0,1]$, $\beta\in(1-\frac{\eta_0}{2},1]$ and $\zeta\in(\vartheta,1]$, we can find $\eta\in(0,\eta_0)$ sufficiently close to $\eta_0$ and $\gamma\in(0,1)$ sufficiently close to $0$, such that the following holds at the same time:
\begin{equation}
    \label{burgerscondition2}
\begin{cases}
\gamma\in(0,\alpha),\\
(1-\gamma)(1-\vartheta)+\zeta>1,\\
\frac{1-\gamma}{2}\eta+\beta>1.\\
\end{cases}\end{equation}
From this we can see that the constraint $\zeta\geq\vartheta$ arises from the $\lambda^{-\vartheta}$ factor in the estimate \eqref{burgersusefulestimate}.

Then we choose $\chi\in(0,\frac{1-\gamma}{2}\eta+\beta-1]$ and set $R:=(\Delta_K^n)^{-\chi}$ in the estimate \eqref{burgerscomplementary}.
A simple computation shows that the powers of $\Delta_K^n$ in the bracket of \eqref{burgerscomplementary} are all larger than one.

Denote by $$\Omega_\nu:=\{\sup_{t\in[0,\tau]}|X_t-\widetilde{X}_t^n|>C (\Delta_K^n)^{1+\alpha-\gamma}+C (\Delta_K^n)^{(1-\gamma)(1-\vartheta)+\zeta}+C(\Delta_K^n)^{\frac{1-\gamma}{2}\eta+\beta-\chi}\}.$$
Upon choosing $\Delta_K^n$ sufficiently small, we deduce that $|X_t-\widetilde{X}_t^n|\leq\Delta_K^n$ for all $t\leq\tau$ on $\Omega\setminus\Omega_\nu$. Since the processes have continuous trajectories, we deduce that we indeed have $$\theta_K\wedge T\leq \tau_K^n$$ on $\Omega\setminus\Omega_\nu$. From this we conclude that, for any $\kappa>0$,
\begin{equation}
   \label{burgersconvergeinprob} 
\mathbb{P}(\sup_{t\in[0,T\wedge \theta_K]} |X_t-\widetilde{X}_t^n|>\kappa)\to 0,\quad n\to\infty.\end{equation}

The previous computations also imply that $\mathbb{P}(\Omega_\nu)\to 0$, so that
\begin{equation}\label{gaagagagagagagagaga}
    \mathbb{P}(\tau_K^n< \theta_K\wedge T)\to 0,\quad n\to\infty.
\end{equation}

\textbf{Step 2. Girsanov estimate.} Consider another auxiliary SPDE
$$\begin{aligned}
d\widetilde{\underline{X}}^n_t&=A \widetilde{\underline{X}}^n_t dt+b^n(t,\widetilde{\underline{X}}^n_t)dt\\&+(-A)^{\vartheta}F^n(t,\widetilde{\underline{X}}^n_t)dt+\lambda (X_t-\widetilde{\underline{X}}_t^n)dt 1_{t\leq\tau} +\sigma^n(t,\widetilde{\underline{X}}^n_t)dW_t,\quad \widetilde{\underline{X}}_0^n=x.\end{aligned}$$ This SPDE has a unique strong solution thanks to the fact that $\tau$ is uniquely defined and all the coefficients are Lipschitz continuous.

Arguing as in \eqref{variationineq}, we deduce that for any $T>0$ we can find a constant $C$ (uniform in $n$) such that
\begin{equation}\label{burgersvariationineq}
d_{TV}(\operatorname{Law}(X^n|_{[0,T]}),\operatorname{Law}(\widetilde{\underline{X}}^n|_{[0,T]}))\leq C T^{1/2}(\Delta_K^n)^\gamma.\end{equation}

\textbf{Step 3. Combining existing bounds.}
The remaining proof of weak uniqueness follows from the same argument as those in the proof of Theorem \ref{theorem1.1}. We only need to combine \eqref{burgersconvergeinprob}, \eqref{burgersvariationineq} and the fact that $\widetilde{X}_t^n=\widetilde{\underline{X}}_t^n$ whenever $t\leq \tau$, and $\mathbb{P}(\tau\leq T)$ can be made arbitrarily small by setting $n\to\infty$ ans $K$ arbitrarily large thanks to \eqref{gaagagagagagagagaga}.
\end{proof}

\section{Long time behaviour}\label{section3!}

\subsection{Choice of a metric and contracting properties}
\label{sec3.145}
In this subsection we give a proof to the following theorem that does not rely on the assumed Lyapunov condition. This choice of metric is inspired by \cite{kulik2020well}.

We consider the distance function $$d_{N,\gamma}(x,y):=(N|x-y|^\gamma)\wedge 1,\quad N\geq 1,\gamma\in[0,1].$$
 The constant $\gamma>0$ and $N>0$ will be chosen in a problem specific way, and is different in every other applications.

Before stating the theorem, we warn the reader that the distance $d_{N,\gamma}$ is not globally comparable to the usual distance on $H$, and in the next result we will choose $N$ and $\gamma$ such that the dynamics is contracting at a \textit{specific} time $t$ with respect to $d_{N,\gamma}$: this result alone does not say anything about the long time behavior of solutions. Indeed it is easy to construct solutions moving away from a fixed region in long time, yet for a specific time $t$ and specific choice of $N,\gamma$ the following \textit{contraction} claim still holds.

\begin{theorem}\label{contractionmain}
 Assume that all the assumptions of Theorem \ref{theorem1.1} are satisfied, and that $b$ and $\sigma$ do not depend on time. For each $x\in H$ denote by $(X_t^x)_{t\geq 0}$ the solution to \eqref{evolutionequation} with initial value $x$. Then for any $D>0$ and fixed $t>0$ we may find some $\gamma\in(0,1)$, some positive integer $N$, and some $\theta\in(0,1)$ depending on $t$ such that
\begin{equation}\label{22111}d_{N,\gamma}(\operatorname{Law}(X_t^x),\operatorname{Law}(X_t^y))\leq \theta  d_{N,\gamma}(x,y),\quad \text{ given }|x|\leq D,|y|\leq D.\end{equation}
and 
\begin{equation}\label{22112}d_{N,\gamma}(\operatorname{Law}(X_t^x),\operatorname{Law}(X_t^y))\leq \theta  d_{N,\gamma}(x,y),\quad  \text{ given } d_{N,\gamma}(x,y)<1,\end{equation} 
where for two Borel probability measures $\mu$ and $\nu$ on the Hilbert space $H$ we define $$
d_{N,\gamma}(\mu,\nu)=\inf_{\lambda\in\mathcal{C}(\mu,\nu)}d_{N,\gamma}(x,y)\lambda(dx,dy)
$$ where $\mathcal{C}(u,v)$ denotes the set of all couplings between $\mu$ and $\nu$.

(Warning again: these estimates hold only at time $t$, and may not hold on any longer time period).

\end{theorem}

This theorem is remarkable as we have not assumed any contracting assumptions on the dynamics $X_t$. It is crucial to see that it is indeed the non-degeneracy of noise that leads to such estimates. If we have a deterministic PDE with more than one invariant measure, the claim of this theorem is false. In the proof we will couple $X_t^x$ to another process with a linear contraction, and control that process and $X_t^y$ via a total variation estimate. The fact that the latter estimate is in terms of total variation distance rather than Wasserstein distance is the central part of the argument, with which we can set $N$ to be very large in the distance $d_{N,\gamma}$, and quantitatively kill the contribution from total variation part. Such a total variation estimate, of course, relies on non-degeneracy of the noise.

The rest of the subsection is devoted to the proof of Theorem \ref{contractionmain}.

\subsubsection{Setting up the coupling}

Fix two initial values $x,y\in H$. To couple the two processes
\begin{equation}\label{eqx}dX_t=AX_t dt+b(X_t)dt+\sigma(X_t)dW_t,\quad X_0=x\end{equation}
and 
\begin{equation}\label{eqy}dY_t=AY_t dt+b(Y_t)dt+\sigma(Y_t)dW_t,\quad Y_0=y,\end{equation}
we introduce a constant $\lambda>0$ and a stopping time $\tau$ whose value will be determined later, and we consider an auxiliary control process
\begin{equation}\label{controlprocess}
d\widetilde{Y}_t=A\widetilde{Y}_t dt+b(\widetilde{Y}_t)dt+\lambda(X_t-\widetilde{Y}_t)dt1_{t\leq\tau}+\sigma(\widetilde{Y}_t)dW_t,\quad \widetilde{Y}_0=y.\end{equation}
Strictly speaking, we also need to approximate $b$ and $\sigma$ by Lipschitz mappings as detailed in Section \ref{subsection2.112}, but we omit this auxiliary approximation procedure as it is exactly the same in the previous proofs of weak uniqueness. Moreover, since weak uniqueness has been proved, the law of $\widetilde{Y}_t$ does not depend on the stopping time $\tau$ whenever $t\leq\tau$, because the law of $\widetilde{Y}_t$ arises as the limit of SPDEs with Lipschitz coefficients, whose distribution is fixed irrespective of the stopping time $\tau$.

We set $$\tau=\tau_{x,y}:=\inf\{t\geq 0:|X_t-\widetilde{Y}_t|\geq 2|x-y|\},$$
and set $\lambda:=|x-y|^{\gamma -1}$ for some $\gamma\in(0,1)$ to be determined later. 

It follows from Proposition \ref{proposition2.3} that for any $T>0$ we can find $C$ such that 
\begin{equation}
    \label{controlgirsanov}
d_{TV}(\operatorname{Law}(Y|_{[0,T]}),\operatorname{Law}(\widetilde{Y}|_{[0,T]}))\leq CT^{1/2} |x-y|^\gamma.\end{equation}

Writing \eqref{eqx} and \eqref{eqy} in mild formulations, we have, for any $t\leq\tau$,
\begin{equation}\label{difference}
\begin{aligned}
|X_t-\widetilde{Y}_t|\leq e^{-\lambda t}|x-y|&+e^{-\lambda t}\left|\int_0^t S(t-s)e^{\lambda s}(b(X_s)-b(\widetilde{Y}_s))ds\right|\\&+e^{-\lambda t}\left|\int_0^t S(t-s)e^{\lambda s}(\sigma(X_s)-\sigma(\widetilde{Y}_s))dW_s\right|.
\end{aligned}\end{equation}
Using the Hölder continuity $\mathbf{H}_2$ and $\mathbf{H}_3$ of $b$ and $\sigma$, definition of the stopping time $\tau$ and Theorem \ref{stochasticinte}, we obtain the following deviation estimate: for each  $\eta\in(0,\frac{1}{2})$, we have that for any suficiently large $m\in\mathbb{N}_+$ we may find some fixed constant $M=M_{m,\eta}>0$ such that 
\begin{equation}\label{fififi}
    \mathbb{P}(\sup_{0\leq t\leq \tau\wedge T}(|X_t-\widetilde{Y}_t|- e^{-\lambda t}|x-y|)\geq  M\lambda^{-1}|x-y|^{\alpha}+M\lambda^{-\frac{1}{2}\eta}|x-y|^{\beta} R)\leq \frac{M}{R^m},\quad R>0.
\end{equation}

Recall we made the choice $\lambda=|x-y|^{\gamma -1}$. We may find  $\gamma>0$ sufficiently small and $\eta\in(0,\eta_0)$ sufficiently close to $\eta_0$ such that
\begin{equation}\label{730f}
\begin{cases}
\gamma\in(0,\alpha)\\
\frac{1-\gamma}{2}\eta+\beta>1.\\
\end{cases}\end{equation} holds, as it did in
\eqref{condition2}.

Then we find some \begin{equation}\label{731f}0<\chi<\frac{1}{2}\min(\alpha-\gamma,\beta+\frac{1-\gamma}{2}\eta-1)\end{equation} and choose furthermore $$R=|x-y|^{-\chi}$$  to obtain
\begin{equation}
    \mathbb{P}(\sup_{0\leq t\leq\tau\wedge T} (|X_t-\widetilde{Y_t}|-e^{-|x-y|^{\gamma -1}t}|x-y|)\geq M_m|x-y|^{1+\chi})\leq M_m |x-y|^{2m\chi},
\end{equation}
where we readily compute that $\lambda^{-1}|x-y|^\alpha= |x-y|^{1-\gamma+\alpha}\leq |x-y|^{2\chi+1}\leq|x-y|^{\chi+1}$ for $|x-y|\leq 1$ and that $\lambda^{-\frac{1}{2}\eta}|x-y|^\beta R\leq |x-y|^{\beta+\frac{1-\gamma}{2}\eta-\chi}\leq |x-y|^{\chi+1}.$

Denote by $$\Omega_\nu:=\{\sup_{0\leq t\leq \tau\wedge T}\{|X_t-\widetilde{Y}_t|-e^{-|x-y|^{\gamma -1}t}|x-y|\}\geq M_m|x-y|^{1+\chi}\},$$ then by choosing $|x-y|$ small enough, on the complementary set $\Omega\setminus\Omega_\nu$ one must have $\sup_{0\leq t\leq\tau_\wedge T} |X_t-Y_t|\leq 2|x-y|$. Thus by continuity of the trajectories we must have \begin{equation}\label{stoppingtimestimate}\tau=\tau_{x,y}\geq T\end{equation}on $\Omega\setminus\Omega_\nu$, so we can deduce that
\begin{equation}\label{1.111}
    \mathbb{P}(\sup_{0\leq t\leq T} (|X_t-\widetilde{Y_t}|-e^{-|x-y|^{\gamma -1}t}|x-y|)\geq M_m|x-y|^{1+\chi})\leq M_m |x-y|^{2m\chi}.
\end{equation}

Now we make a choice of $m\in\mathbb{N}_+$ such that \begin{equation}\label{limitofm}m\chi \geq \gamma,\end{equation}
and fix the choice of the constant $M:=M_m$ in the rest of the argument.

Since $$e^{-\nu^{\gamma -1}T}+M\nu^\chi\to 0\quad \text{ as   }\nu\to 0,$$ 
we will choose some sufficiently small  $\nu_0>0$ and assume that $|x-y|\leq \nu_0$ to obtain, as a consequence of \eqref{1.111}, \eqref{stoppingtimestimate} and \eqref{limitofm}: 
\begin{equation} \label{deviationestimate11}
    \mathbb{P}(|X_T-\widetilde{Y}_T|\geq \frac{1}{2}|x-y|)\leq M |x-y|^{2\gamma}.
\end{equation}

Since probability is upper bounded by 1, we restate the above estimate as
\begin{equation} \label{deviationestimate}
    \mathbb{P}(|X_T-\widetilde{Y}_T|\geq \frac{1}{2}|x-y|)\leq M |x-y|^{2\gamma}\wedge 1.
\end{equation}

We can also derive the following estimate which will be useful later:
\begin{equation} \label{longtime}
    \mathbb{P}(\sup_{0\leq t\leq T}|X_t-\widetilde{Y}_t|\geq 2|x-y|)\leq M |x-y|^{2\gamma}\wedge 1.
\end{equation}

\subsubsection{The case of close enough initial values}\label{Section3.2}
In this section we consider initial values $x,y\in H$ that are close in the sense that $$d_{N,\gamma}(x,y)<1.$$
We combine \eqref{deviationestimate} with the control \eqref{controlgirsanov}to deduce a contraction estimate for $|X_t-Y_t|$ under the distance $d_{N,r}.$ The proof is very similar to Proposition 5.1 of \cite{kulik2020well}.

Recall that we have derived \eqref{deviationestimate} assuming $x$ and $y$ are close enough: for some $\nu_0>0$ sufficiently small, $|x-y|\leq\nu_0$. Thus we choose $N$ large enough \footnote{once we have fixed a value of $\gamma>0$} such that $d_{N,\gamma}(x,y)=1$ for any pair $(x,y)$ with $|x-y|\geq\nu_0.$ We will work under this assumption throughout the proof without further mentioning.

For any two initial values $x,y\in H$, consider the SPDEs $X_t$ \eqref{eqx}, $Y_t$ \eqref{eqy}, as well as the control process $\widetilde{Y}_t$ \eqref{controlprocess}.

Consider a coupling $(\xi_1,\xi_2)$ where $\xi_1$ has law $X_t$, $\xi_2$ has law $\widetilde{Y}_t$ and that $(X_t,\widetilde{Y}_t)$ satisfy the contraction estimate \eqref{deviationestimate} with $t=T$.

Consider another coupling $(\xi_2,\xi_3)$ where $\xi_2$ has law $\widetilde{Y}_t$ and $\xi_3$ has law $Y_t$, such that $(\xi_2,\xi_3)$ realizes the total variance distance: $\mathbb{P}(\xi_2\neq \xi_3)= d_{TV}(\operatorname{Law}(Y|_{[0,T]});\operatorname{Law}(\widetilde{Y}|_{[0,T]}))$. Then by \eqref{controlgirsanov}, $$\mathbb{P}(\xi_2\neq \xi_3)= d_{TV}(\operatorname{Law}(Y|_{[0,T]});\operatorname{Law}(\widetilde{Y}|_{[0,T]}))\leq CT^{1/2}|x-y|^\gamma.$$
We may now form a coupling $(\xi_1,\xi_3)$ of $(X_t,Y_t),$ possibly on a different probability space, via the coupling $(\xi_1,\xi_2)$ and $(\xi_2,\xi_3)$ and obtain the following estimate:
$$\begin{aligned}
\mathbb{E}d_{N,\gamma}(\xi_1,\xi_3)&\leq \mathbb{E} d_{N,\gamma}(\xi_1,\xi_2)+\mathbb{P}(\xi_2\neq\xi_3)\\&\leq \mathbb{E}d_{N,\gamma}(\xi_1,\xi_2)1_{\{|\xi_1-\xi_2|\leq\frac{1}{2} |x-y|\}}+\mathbb{P}(|\xi_1-\xi_2|\geq \frac{1}{2}|x-y|)+\mathbb{P}(\xi_2\neq\xi_3).\end{aligned}
$$

Since we assume that $$d_{N,\gamma}(x,y)=N|x-y|^\gamma<1,$$ and therefore $N2^{-\gamma} |x-y|^\gamma<1,$ we combine the various estimates to deduce that 
\begin{equation}\label{goodgifts}\begin{aligned}\mathbb{E}d_{N,\gamma}(\xi_1,\xi_3)&\leq N2^{-\gamma} |x-y|^\gamma +M|x-y|^{2\gamma}+C|x-y|^\gamma\\ \leq& d_{N,\gamma}(x,y)\left(2^{-\gamma}+\frac{M|x-y|^{2\gamma}\wedge 1}{N|x-y|^\gamma}+\frac{C}{N}\right).\end{aligned}\end{equation}
Note that $\sup_{s\geq 0}\frac{M s^{2\gamma}\wedge 1}{s^\gamma}<\infty.$
Therefore, upon choosing $N$ large enough, we may find some $\theta_1\in(2^{-\gamma},1)$ such that, for some given $N_0$ depending on $\gamma$ and $\theta_1$, we have
\begin{equation}\mathbb{E}d_{N,\gamma}(\xi_1,\xi_3)\leq \theta_1 d_{N,\gamma}(x,y),\quad \text{ for all } N\geq N_0,\quad \text{all } d_{N,\gamma}(x,y)<1.\end{equation}
The constant $\theta_1$ depends on the Hölder continuity index of various coefficients, and also depends on the final time $t$, but does not depend on the initial value $x$ and $y$.

(From these computations, one can see why we really need a total variation distance bound between $\xi_2,\xi_3$, rather than merely a Wasserstein distance type bound, to guarantee that we can take $\theta_1\in(0,1)$. The total variation bound gives a coupling such that they are identical with high probability. Therefore we are free to modify the distance on $H$ by setting $N$ large, so that the contribution from this total variation part would be rather small. If the distance bound between $\xi_2,\xi_3$ were again of Wasserstein type, then setting $N$ large in the distance will male $d(\xi_1,\xi_2)$ and $d(\xi_2,\xi_3)$ grow at equal rate and thus we cannot get rid of the second term effectively. We stress again that such a total variation estimate originates from non-degeneracy of the noise, which is the only important part needed in this proof.)

By definition of the Wasserstein distance,  $$
d_{N,\gamma}\left(\operatorname{Law}(X_t),\operatorname{Law}(Y_t)\right)\leq \mathbb{E}d_{N,\gamma}(\xi_1,\xi_3),$$
so we conclude that \begin{equation}d_{N,\gamma}\left(\operatorname{Law}(X_t),\operatorname{Law}(Y_t)\right)
\leq \theta_1 d_{N,\gamma}(x,y),\quad N\geq N_0,\quad d_{N,\gamma}(x,y)<1.\end{equation}
Note that the choice of $N$ and $\gamma$ depend on the final time $t>0$. 

To unify notations in this paper, we rewrite this estimate as
\begin{equation}\label{localcontraction}d_{N,\gamma}\left(\operatorname{Law}(X_t^x),\operatorname{Law}(X_t^y)\right)
\leq \theta_1 d_{N,\gamma}(x,y),\quad N\geq N_0,\quad d_{N,\gamma}(x,y)<1.\end{equation}

One could also derive in the same way from \eqref{longtime} the following estimate
\begin{equation}\label{longtimespider}
    \sup_{0\leq s\leq t} d_{N,\gamma}\left(\operatorname{Law}(X_s^x),\operatorname{Law}(X_s^y)\right)
\leq 3 d_{N,\gamma}(x,y),\quad N\geq N_0,\quad d_{N,\gamma}(x,y)<1.
\end{equation}

\subsubsection{The case of distant initial values}

We finally derive the estimate for initial values $x,y$ such that $d_{N,\gamma}(x,y)=1,$ but that for some fixed $R>0$, $|x|\leq R$ and $|y|\leq R$.

We begin with a lemma that describes the possibility of visiting bounded subsets of $H$, which is similar to Proposition 5.2 of \cite{kulik2020well} and to the proof of various estimates here. 

\begin{lemma}\label{lemma3.1}
For each $x\in H$ and $\lambda>0$, denote by $X_t^{\lambda,x}$ the solution to the SPDE 
$$dX_t=A X_t dt-\lambda X_t dt+b(X_t)dt+\sigma(X_t)dW_t,\quad X_0=x,$$ where $b$ and $\sigma$ satisfy the assumptions $\mathbf{H}_1$ to $\mathbf{H}_5$.

Then for any $D>0$ and $\delta>0$, there exists $\lambda$ sufficiently large (depending only on $D$, $\delta$, $t$ and the constants in assumption $\mathbf{H}_1$ to $\mathbf{H}_5$) such that  
$$\inf_{|x|\leq D} \mathbb{P}(|X_t^{\lambda,x}|\leq\delta)\geq\frac{1}{2}.$$
\end{lemma}

\begin{proof}
Consider the stopping time $\tau_D:=\inf\{t\geq 0: |X_t^{\lambda,x}|\geq 2D\}.$ By the linear growth property $\mathbf{H}_5$, we may find some $M>0$ such that $$|b(X_t^{\lambda,x})|\leq M(1+D),\quad |\sigma(X_t^{\lambda,x})|_{\mathcal{L}(H_0,H)}\leq M(1+D)$$ whenever $t\leq\tau_D$. Arguing with the mild formulation and estimating each term explicitly, using also Theorem \ref{stochasticinte} with $\eta=\frac{\eta_0}{2}$, we claim that
$$\mathbb{P}\left(\sup_{s\in[0,t\wedge\tau_D]} \{|X_s^{\lambda,x}|- e^{-\lambda s}|x|\}\geq \lambda^{-1}M(1+D)+\lambda^{-\frac{\eta_0}{4}}M(1+D)R\right)\leq \frac{C}{R^2},\quad R\geq 2.$$
Choosing $R=\lambda^{\frac{\eta_0}{8}}$ for $\lambda$ large, we see that 
\begin{equation}\label{event}\mathbb{P}\left(\sup_{s\in[0,t\wedge\tau_D]} \{|X_s^{\lambda,x}|- e^{-\lambda s}|x|\}\geq \lambda^{-1}M(1+D)+ \lambda^{-\frac{\eta_0}{8}}M(1+D)\right)\leq C \lambda^{-\frac{\eta_0}{4}}.\end{equation}
Arguing as in the previous proofs, we see that when $\lambda$ is sufficiently large, in the complement of the event in \eqref{event} one must have $\tau_D\geq t$. Now if we choose $\lambda$ sufficiently large, we can ensure that 
$\mathbb{P}(|X_t^{\lambda,x}|\leq \delta)\geq \frac{1}{2},$ uniformly over the initial value $|x|\leq D$. This completes the proof.
\end{proof}

In the following we fix this choice of $\lambda>0$. Via direct moment estimates, we may find 
$$C_1(\lambda):=\sup_{|x|\leq D}\mathbb{E}[\sup_{s\in[0,t]}|X^{\lambda,x}_s|^2]<\infty.$$ By arguing via Girsanov transform, we compute the relative entropy between solutions $X^x$ and $X^{\lambda,x}$ as follows:
\begin{equation}
    \label{relativedifference}\sup_{|x|\leq D} H_{\text{rel}}(\operatorname{Law}(X^{\lambda,x}|_{[0,t]})\mid\operatorname{Law}(X^{x}|_{[0,t]}))\leq \frac{\lambda t}{2}C_1(\lambda)\sup_{x\in H}|\sigma(x)^{-1}|:=C_2<\infty,\end{equation} where $H_{\text{rel}}$ denotes the relative entropy on path space $C([0,T];H).$

 We will use a powerful formula on relative entropy (see \cite{kulik2020well}, Appendix A or \cite{butkovsky2020generalized}, Appendix A): for two probability measures $\mu$ and $\nu$ on a common measure space $(X,\mathcal{B})$, and for any set $A\in\mathcal{B}$, we have for each $N\in\mathbb{N}_+$
$$\nu(A)\geq\frac{1}{N}\mu(A)-\frac{H_{\text{rel}}(\mu\mid\nu)+\log 2}{N\log N}.$$

In our case we consider $A:=\{\omega\in \mathcal{C}([0,t];H): |\omega_t|\leq\delta\},$ $\nu=\operatorname{Law}(X^x|_{[0,t]})$ and $\mu=\operatorname{Law}(X^{\lambda,x}|_{[0,t]}).$ Combining Lemma \ref{lemma3.1} and estimate \eqref{relativedifference}, we deduce that 
$$\inf_{|x|\leq D}\mathbb{P}(|X_t^x|\leq\delta)\geq\frac{1}{2N}-\frac{C_2+\log 2}{N\log N}.$$ 
Now we choose $N=\exp(4C_2+4\log 2)$ and deduce that 
\begin{equation}\label{finalconclusion}\inf_{|x|\leq D}\mathbb{P}(|X_t^x|\leq\delta)\geq\frac{1}{L(D,\delta)},\end{equation} where $L(D,\delta):=4\exp(4C_2+4\log 2)$ is a constant that only depends on $D$, $\delta$, $t$ and the constants in hypothesis $\mathbf{H}_1-\mathbf{H}_5.$

\subsubsection{Concluding the proof of Theorem \ref{contractionmain}}
\begin{proof}
Consider two initial values $x,y\in H$ such that $|x|\leq D$, $|y|\leq D$ and $d_{N,\gamma}(x,y)=1$, where the values of $N$ and $\gamma$ have been fixed in Section \ref{Section3.2}.

Consider the solutions $X_t^x$ and $X_t^y$. If $$|X_t^x|\leq\frac{1}{2^{1+1/\gamma}N^{1/\gamma}}\quad\text{ and }\quad |X_t^y|\leq \frac{1}{2^{1+1/\gamma}N^{1/\gamma}},$$ then $d_{N,\gamma}(X_t^x,X_t^y)\leq\frac{1}{2}.$ Otherwise we upper bound $d_{N,\gamma}(X_t^x,X_t^y)$ by 1. Using an independent coupling of $X_t^x$ and $X_t^y$, we compute that
\begin{equation}\label{largecontraction}
\begin{aligned}
    d_{N,\gamma}(\operatorname{Law}(X_t^x),\operatorname{Law}(X_t^y))&\leq d_{N,\gamma}(X_t^x,X_t^y)\\&\leq 1-\frac{1}{2}\left(\inf_{|z|\leq D}\mathbb{P}(|X_t^z|\leq\frac{1}{2^{1+1/\gamma}N^{1/\gamma}})\right)^2<1.\end{aligned}
\end{equation}
 Take $\theta$ to be the maximum of the constant $\theta_1$ in \eqref{localcontraction} and the constant in the right hand side of \eqref{largecontraction}, we conclude that 
 \begin{equation}d_{N,\gamma}\left(\operatorname{Law}(X_t),\operatorname{Law}(Y_t)\right)
\leq \theta d_{N,\gamma}(x,y),\quad |x|\leq D,|y|\leq D.\end{equation}
This finishes the proof.
\end{proof}

\subsection{Lyapunov function and exponential ergodicity}
\subsubsection{Existence and uniqueness of invariant measure}
We will need to use some abstract ergodic results from \cite{hairer2011asymptotic}, and is also related to \cite{hairer2006ergodicity}, \cite{hairer2011theory}, \cite{hairer2002exponential} and \cite{hairer2011yet}.
We begin with some notations from \cite{hairer2011asymptotic}. 
We say a function $d:H\times H\to [0,1]$ is distance-like if it is symmetric, lower-semicontinuous and such that $d(x,y)=0$ implies $x=y$. We extend $d$ to a positive function $\mathcal{M}_1(H)\times\mathcal{M}_1(H)\to\mathbb{R}_+$, with $\mathcal{M}_1(H)$ the set of all Borel probability measures on $H$, by 
\begin{equation}\label{wessagegag}d(\mu,\nu)=\inf_{\pi\in\mathcal{C}(\mu,\nu)}\int_{H\times H}d(x,y)\pi(dx,dy),\end{equation}
where $\mathcal{C}(\mu,\nu)$ consists of all the couplings of $\mu$ and $\nu$. This is exactly the 1-Wasserstein distance when $d$ is a metric on $H$.

Given a Markov operator $\mathcal{P}$ on $H$, we let $\mathcal{P}(x,\cdot)$ denote the transition probability distribution of $\mathcal{P}$ from the initial value $x\in H,$
Recall that for a Markov operator $\mathcal{P}$ on a measurable space $H$, we say a set $A\subset H$ is $d$-small if for some some $\epsilon>0,$
$$d(\mathcal{P}(x,\cdot),\mathcal{P}(y,\cdot))\leq 1-\epsilon$$
for any $x,y\in A$.
Now we are in the position to prove Theorem \ref{harristheorem}. We first show that there exists at most one invariant measure.

\begin{proof} Let $\mathcal{P}_t$ denote the Markov semigroup generated by $(X_t)$.
We check all the assumptions in Theorem 4.8 of \cite{hairer2011asymptotic} are satisfied. Fix some $t>0$ and set $D:=\sup\{|x|:x\in H, V(x)\leq 4K_V\}$ \footnote{The supremum $D$ is finite thanks to the assumption $V(x)\to\infty$ as $|x|\to\infty$.} in the setting of Theorem \ref{harristheorem}. Then we may find a choice of $N$ and $\gamma$ such that the distance function $d_{N,\gamma}$ satisfies
\begin{itemize}
    \item $\mathcal{P}_{t}$ is contracting in the sense \eqref{22112} for initial values $x,y\in H$ with $d_{N,\gamma}(x,y)<1$. 
    \item The sub-level set $\{x\in H:V(x)\leq 4K_V\}$ is $d$-small for $\mathcal{P}_{t}$.
\end{itemize}
Indeed, the first claim is fulfilled by \eqref{22112} and the second claim is fulfilled by \eqref{22111}.

Then by Theorem 4.8 of \cite{hairer2011asymptotic}, $\mathcal{P}_t$ can have at most one invariant measure, and if we define $\widetilde{d}(x,y)=\sqrt{d_{N,\gamma}(x,y)(1+V(x)+V(y))}$, there exists $t_*>0$ such that 
$$\widetilde{d}(\mathcal{P}_{t_*}\mu,\mathcal{P}_{t_*}\nu)\leq \frac{1}{2}\widetilde{d}(\mu,\nu)$$
for any probability measures $\mu,\nu\in\mathcal{M}_1(H)$.

By the elementary inequality 
$$1\wedge |x-y|^\gamma\leq 1\wedge N|x-y|^\gamma\leq N(1\wedge |x-y|^\gamma),$$
we may choose a larger $t_*$ and redefine $\widetilde{d}(x,y)=\sqrt{(1\wedge |x-y|^\gamma)(1+V(x)+V(y))}$ so that we still have
\begin{equation}\label{iterativeexist}\widetilde{d}(\mathcal{P}_{t_*}\mu,\mathcal{P}_{t_*}\nu)\leq \frac{1}{2}\widetilde{d}(\mu,\nu).\end{equation}

The last step is to prove there exists an invariant measure $\pi$ for $\mathcal{P}$. Fix a probability measure $\mu\in\mathcal{M}_1(H)$ such that $\int Vd\mu<\infty$. Applying \eqref{iterativeexist} iteratively, we get
 $$\widetilde{d}(\mathcal{P}_{(n+1)t_*}\mu,\mathcal{P}_{nt_*}\mu)\leq\frac{1}{2^n} \widetilde{d}(\mathcal{P}_{t_*}\mu,\mu).$$
 Since $\widetilde{d}$ dominates $d_0:=1\wedge |x-y|$, and the latter is a complete metric on $H$, we deduce that the Wasserstein distance associated with $\widetilde{d}$ defines a complete metric for probability measures on $H$ that integrates $V.$ Thus there exists a probability measure $\mu_\infty$ such that $\widetilde{d}(\mathcal{P}_{nt_*}\mu,\mu_\infty)$ converges to $0$ as $n\to\infty$. One readily checks that $\widetilde{d}(\mathcal{P}_{t_*}\mu_\infty,\mu_\infty)=0$, so that $\mu_\infty$ is invariant for $\mathcal{P}_{t_*}$.
 
 Finally, define 
 $$\mu_*(A)=\frac{1}{t_*}\int_0^{t_*}(\mathcal{P}_s\mu_\infty)(A)ds$$ for each measurable $A\subset H$. One sees that $\mathcal{P}_t\mu_*=\mu_*$ for any $t>0$.
\end{proof}

\subsubsection{Convergence rate: finishing the proof of Theorem \ref{harristheorem}}  As the idea of proof is similar to Theorem 2.4 of \cite{butkovsky2014subgeometric} and Theorem 2.3 of\cite{kulik2020well}, we only give a sketch of its main ideas.

\begin{proof}
First consider the discrete-time Markov chain $(X_{t_0n})_{n\in\mathbb{N}_+}$. Set $t=t_0$ in Theorem \ref{contractionmain}, we can fix a choice of the distance $d=d_{N,\gamma}$, $N\in\mathbb{N}_+$, $\gamma\in(0,1)$ such that the sub-level set $V^{-1}([0,4C_V])$ is $d$-small. 
We will check that we can apply Theorem 4.5.2 of \cite{kulik2017ergodic}, using the same step as in \cite{kulik2020well}, Section 5.3. Write $K_{V,\ell}:=\{x:V(x)\leq\ell\}$. First, the choice $d=d_{N,\gamma}$ and $B=K_{V,\ell}\times K_{V,\ell}$ justifies condition I of \cite{kulik2017ergodic}, Theorem 4.5.2. Second, the recurrence condition R (i),(ii) in \cite{kulik2017ergodic}, Theorem 4.5.2. holds true with $W(x,y)=V(x)+V(y)$ and $\lambda(t)=\exp(ct)$ (this is a routine but a bit technical application of Lyapunov property, see \cite{kulik2020well}, Section 5.3 for how this step is checked). 
We take $p=\epsilon^{-1}$ and $q=(1-\epsilon)^{-1}$, and define $d_{M,\gamma,p}(x,y)=d_{N,\gamma}(x,y)^{1/p}.$ Then the additional assumption (4.5.8) in \cite{kulik2017ergodic}, Theorem 4.5.2 is verified since $d_{N,\gamma,p}$ is bounded by 1. Applying \cite{kulik2017ergodic}, Theorem 4.5.2
we deduce the existence of an invariant measure $\pi$ and that, for some $c,C>0$,
\begin{equation}\label{geometricmain}d_{N,\gamma,p}(\operatorname{Law}(X^x_{nt_0}),\pi)\leq \frac{C}{(e^{cnt_0})^{1-\epsilon}}(V(x))^{1-\epsilon},\quad n\in\mathbb{N}_+.\end{equation}

Moreover, using \eqref{longtimespider} and that the distance $d_{N,\gamma}$ is upper bounded by 1, we further conclude that for all $x,y\in H$, 
\begin{equation}\label{contractsmall}
d_{N,\gamma}(\operatorname{Law}(X_t^x,X_t^y))\leq  4 d_{N,\gamma}(x,y),\quad 0\leq t\leq t_0.\end{equation}

Combining the Markov property, \eqref{geometricmain} and \eqref{contractsmall}, using also
$$d_{N,\gamma}(x,y)\leq d_{N,\gamma,p}(x,y),$$
we deduce that for some $\xi>0$,
$$d_{N,\gamma}(\operatorname{Law}(X_t^x),\pi)\leq \frac{C}{({e^{\xi t}})^{1-\epsilon}} (V(x))^{1-\epsilon},\quad \text{for all }x\in H,t\geq 0.$$

We complete the proof noting that $$d_{\gamma'}(x,y)\leq d_{\gamma}(x,y)\leq d_{N,\gamma}(x,y)$$
for any $0<\gamma\leq\gamma'\leq 1$ and $N\in\mathbb{N}_+.$
\end{proof}

\subsection{Proof of Theorem \ref{burgerstheorem2}}
We now prove Theorem \ref{burgerstheorem2}. The argument is the same as the previous one, and we only outline where changes are needed when deriving various estimates.
 Arguing as in \eqref{finalconclusion}, we deduce similarly that
\begin{equation}\label{burgersfinalconclusion}\inf_{|x|\leq D}\mathbb{P}(|X_t^x|\leq\delta)\geq\frac{1}{L(D,\delta)}.\end{equation}
Then use an independent coupling, we deduce as in \eqref{largecontraction} that 
\begin{equation}\label{burgerslargecontraction}
\begin{aligned}
    d_{N,\gamma}(\operatorname{Law}(X_t^x),\operatorname{Law}(X_t^y))&\leq d_{N,\gamma}(X_t^x,X_t^y)\\&\leq 1-\frac{1}{2}\left(\inf_{|z|\leq D}\mathbb{P}(|X_t^z|\leq\frac{1}{2^{1+1/\gamma}N^{1/\gamma}})\right)^2<1.\end{aligned}
\end{equation}
for any $(x,y)\in H\times H$ with $|x|\leq D$ and $|y|\leq D$.

For the generalized coupling, fix two initial values $x,y\in H$. Consider the processes
\begin{equation}\label{burgerseqx}dX_t=AX_t dt+b(X_t)dt+(-A)^{1/2}F(X_t)dt+\sigma(X_t)dW_t,\quad X_0=x,\end{equation}
\begin{equation}\label{burgerseqy}dY_t=AY_t dt+b(Y_t)dt+(-A)^{1/2}F(Y_t)dt+\sigma(Y_t)dW_t,\quad Y_0=y,\end{equation}
where we set $$\tau=\tau_{x,y}:=\inf\{t\geq 0:|X_t-\widetilde{Y}_t|\geq 2|x-y|\},$$
and set $\lambda:=|x-y|^{\gamma -1}$ for some $\gamma\in(0,1)$ to be determined later. 
Consider also
\begin{equation}\label{burgerscontrolprocess}
d\widetilde{Y}_t=A\widetilde{Y}_t dt+b(\widetilde{Y}_t)dt+(-A)^{1/2}F(\widetilde{Y}_t)dt+\lambda(X_t-\widetilde{Y}_t)dt1_{t\leq\tau}+\sigma(\widetilde{Y}_t)dW_t,\quad \widetilde{Y}_0=y.\end{equation}

Since weak uniqueness has been proved, the law of $\widetilde{Y}_t$ does not depend on the stopping time $\tau$ whenever $t\leq\tau$, because the law of $\widetilde{Y}_t$ arises as the limit of SPDEs with Lipschitz coefficients, whose distribution is fixed irrespective of the stopping time $\tau$.

Then arguing as in \eqref{controlgirsanov}, we deduce that for any $T>0$ we can find $C$ such that 
\begin{equation}
    \label{burgerscontrolgirsanov}
d_{TV}(\operatorname{Law}(Y|_{[0,T]}),\operatorname{Law}(\widetilde{Y}|_{[0,T]}))\leq CT^{1/2} |x-y|^\gamma.\end{equation}

Then arguing through the mild formulation of $X_t$ and $\widetilde{Y}_t$ as in \eqref{difference}, 
using the Hölder continuity $\mathbf{H}_2$, $\mathbf{H}_3$ and $\mathbf{H}_6$ of $b$, $\sigma$ and $F$, the definition of the stopping time $\tau$ and Theorem \ref{stochasticinte}, we obtain as in \eqref{fififi} the following deviation estimate: for each sufficiently large $m\in\mathbb{N}_+$ and $\eta\in(0,\frac{1}{2})$, we may find some fixed constant $M=M_{m,\eta}>0$ such that for any $R>0$,
\begin{equation}\label{burgersfififi}
    \mathbb{P}(\sup_{0\leq t\leq \tau\wedge T}\{|X_t-\widetilde{Y}_t|- e^{-\lambda t}|x-y|\}\geq  M(\lambda^{-1}|x-y|^{\alpha}+\lambda^{-\frac{1}{2}}|x-y|^\zeta+\lambda^{-\frac{1}{2}\eta}|x-y|^{\beta} R))\leq \frac{M}{R^m},
\end{equation}
where we also used the auxiliary estimate \eqref{burgersusefulestimate}. Recall our choice $\lambda=|x-y|^{\gamma -1}$. We now choose $\gamma>0$ sufficiently small such that \eqref{burgerscondition2} is satisfied. 

Then we find some  $0<\chi<\min(\alpha-\gamma,\beta+\frac{1-\gamma}{2}\eta-1,\zeta+\frac{1-\gamma}{2}-1)$ and set furthermore $$R=|x-y|^{-\chi}.$$ 
Arguing as in \eqref{1.111} and \eqref{deviationestimate}, we conclude that 
\begin{equation} \label{burgersdeviationestimate}
    \mathbb{P}(|X_T-\widetilde{Y}_T|\geq \frac{1}{2}|x-y|)\leq M |x-y|^{2\gamma}\wedge 1.
\end{equation}
Having worked out this short-time estimate, we can construct the distance $d_{N,\gamma}$ as desired. Arguing as in \eqref{goodgifts}, we obtain as in \eqref{localcontraction} that for a choice of $N_0$ and $\gamma$ depending on $t$,
\begin{equation}\label{burgerslocalcontraction}d_{N,\gamma}\left(\operatorname{Law}(X_t^x),\operatorname{Law}(X_t^y)\right)
\leq \theta_1 d_{N,\gamma}(x,y),\quad N\geq N_0,\quad d_{N,\gamma}(x,y)<1.\end{equation}
This completes the proof of Theorem \ref{contractionmain} in the setting of a Burgers type non-linearity, and the whole proof of Theorem \ref{burgerstheorem2} now finishes.

\section{Banach space example: reaction diffusion equations}
\label{section4}
In this section we follow some analytic arguments in \cite{cerrai2003stochastic}. For any $p>0$ denote by $|\cdot|_p$ the norm of $L^p([0,1];\mathbb{R}^r).$ For $\epsilon>0$ define a norm $|\cdot|_{\epsilon,p}$ and the Sobolev space $W^{\epsilon,p}([0,1];\mathbb{R}^r)$ via
$$
|x|_{\epsilon,p}:=|x|_p+\sum_{i=1}^r\int_{[0,1]^2}\frac{|x_i(\xi)-x_i(\eta)|^p}{|\xi-\eta|^{\epsilon p+1}}d\xi d\eta.
$$
We have the following estimates:\begin{equation}\label{4.1123}
    |\mathcal{S}(t)x|_{\epsilon,p}\leq C(t\wedge 1)^{-\frac{\epsilon}{2}}|x|_p,\quad x\in L^p([0,1];\mathbb{R}^r).
\end{equation}

Throughout this section, $H_0$ denotes the Banach space $\mathcal{C}([0,1];\mathbb{R}^r)$, which is the $r$-fold direct sum of $\mathcal{C}([0,1];\mathbb{R})$ with the norm
$$\left|x\right|_{H_0}:=\left(\sum_{i=1}^r \|x_i\|_{\mathcal{C}([0,1];\mathbb{R})}^2\right)^{\frac{1}{2}}.
$$

We prove the following analogue of Theorem \ref{stochasticinte}, where we replace the Hilbert space $H$ by the Banach subspace $H_0$. This reduction is crucial for handling interaction term of polynomial growth. The $\lambda=0$ case of the following theorem is covered in \cite{cerrai2003stochastic}, Theorem 4.2.

\begin{theorem}\label{4.1fuc}
    Fix any $\eta\in(0,\frac{1}{2})$ and recall the notion 
    $$
\Gamma(t)=\int_0^t \mathcal{S}(t-s)e^{-\lambda(t-s)}\Phi(s)dW_s.
    $$ Consider any $H_0$-valued adapted process $\Phi$.
    Then there exists a sufficiently large constant $p_*$ such that for any $p>p_*$, we have
    \begin{equation}
     \mathbb{E}\sup_{0\leq t\leq T}  |\Gamma(t)|_{H_0}^p\leq c_{\eta,p,T}\lambda^{-\frac{\eta p}{2}}\mathbb{E}\sup_{0\leq t\leq T}|\Phi(t)|_{H_0}^p,
    \end{equation}
    where $c_{\eta,p,T}$ is a positive constant  continuous and growing in $T$ and satisfies $c_{\eta,p,0}=0$.

    We can also take $\lambda=0$ in the definition of $\Gamma(t)$ and obtain the estimate
     \begin{equation}
     \mathbb{E}\sup_{0\leq t\leq T}  |\Gamma(t)|_{H_0}^p\leq c_{\eta,p,T}\mathbb{E}\sup_{0\leq t\leq T}|\Phi(t)|_{H_0}^p.
    \end{equation}
    
\end{theorem}
(Note: since the interval $(0,1)$ is compact, any $H_0$-valued process $\Phi$ can be identified as a $\mathcal{L}(H_0,H)$-valued process $\widetilde{\Phi}$ defined by $u\in H_0\mapsto \Phi\cdot u\in H$ in such a way that for some universal constant $C>0$, $|\widetilde{\Phi}|_{\mathcal{L}(H_0,H)}\leq C|\Phi|_{H_0}$.)
\begin{proof}

We will be sketchy in the proof as the proof is a combination of Theorem \ref{stochasticinte} and\cite{cerrai2003stochastic}, Theorem 4.2. By the Sobolev estimate \eqref{4.1123} and the stochastic factorization lemma \eqref{2.2maim2.2}, we deduce that for $\alpha>1/p$ and $\epsilon<2(\alpha-1/p)$, using Hölder's inequality:
\begin{equation}\begin{aligned}
     \mathbb{E}\sup_{0\leq t\leq T}  |\Gamma(t)|_{\epsilon,p}^p&\leq C_\alpha \int_0^T ((T-r)\wedge 1)^{\alpha-\frac{\epsilon}{2}-1}|Z(r)|_p dr\\&\leq 
     C_\alpha (\int_0^{T}(r\wedge 1)^{\frac{p}{p-1}(\alpha-2\epsilon-1)}dr)^{\frac{p-1}{p}} \mathbb{E}\sup_{0\leq t\leq T}|Z(t)|_{L^p([0,T]\times [0,1],\mathbb{R}^r)}^p,
\end{aligned}\end{equation}
where 
\begin{equation}
    Z(t)=\int_0^t(t-s)^{-\alpha}\mathcal{S}^\lambda(t-s)\Phi(s)dW_s.
\end{equation}
For any sufficiently small $\alpha_*$ we can find $p_*$ such that when $p\geq p_*$ we have $ \alpha_*>\frac{3}{2p}<1.$ This is useful for Sobolev embedding: when $\epsilon>1/p$ we have $\Gamma(t)\in \mathcal{C}([0,T];\mathbb{R}^r).$
    Now we bound $Z(t)_{L^p([0,T]\times [0,1];\mathbb{R})}$. In this case we have, by BDG inequality, that for any $\xi\in(0,1)$,
    $$\begin{aligned}
\mathbb{E}|Z(r)(\xi)|^p&\leq C\mathbb{E}\left(\int_0^r (r-r')^{-2\alpha_*}\sum_{k=1}^\infty |\mathcal{S}^\lambda(r-r')\Phi(r')e_k(\xi)|^2 dr'\right)^{\frac{p}{2}}
\\&\leq C\mathbb{E}\left(\int_0^r (r-r')^{-2\alpha_*}e^{-2\lambda(r-r')}\sum_{k=1}^\infty |\mathcal{S}(r-r')\Phi(r')e_k(\xi)|^2 dr'\right)^{\frac{p}{2}},\\&
\leq C\mathbb{E}\left(\int_0^r (r-r')^{-(2\alpha_*+\frac{1}{2})}e^{-2\lambda(r-r')}\sup_{s\leq r'}|\Phi(s)|_{H_0}dr'\right)^{\frac{p}{2}}
    \end{aligned}$$
where for the last inequality we have skipped many computations: these computations correspond to intermediate steps in the proof of \cite{cerrai2003stochastic}, Theorem 4.2, and we take $\zeta=1$ and $d=1$ in that proof. Applying the elementary inequality
$$
e^{-2\lambda(r-r')}\leq \frac{C_\tau}{\lambda^\tau(r-r')^\tau}
$$
and choosing $\tau$ sufficiently small such that $2\alpha_*+\frac{1}{2}+\tau<1$ so the integral is finite, we get the desired $\lambda^{-\frac{\eta p}{2}}$ leading term for any $\eta\in(0,\frac{1}{2})$. The rest of the proof shall be completed following the steps in \cite{cerrai2003stochastic}, Theorem 4.2. The continuity in $T$ of $c_{\eta,p,T}$ and that $c_{\eta,p,0}=0$ can be readily read off from the computations. The case $\lambda=0$ can be found in in \cite{cerrai2003stochastic}, Theorem 4.2 and is simpler to prove.
    
\end{proof}

We will also use the following estimate from \cite{cerrai2003stochastic}, Lemma 5.4.

\begin{lemma}\label{1028} Under the assumptions of Theorem \ref{theorem1.8}, assuming moreover that $f_i$ and $g_i$ are globally Lipschitz continuous, then for any $T>0$ and $p\geq 1$ we have the estimate
    \begin{equation}
        \mathbb{E}[\sup_{o\leq t\leq T}|X_t|_{H_0}^p]\leq C(p,T)(1+|x|_{H_0})^p,
    \end{equation}
    where the constant $C(p,T)$ depends on the constant $C_5$ in the dissipative estimate, $C_4$ in the polynomial growth estimate and the exponent $p$, but is independent of the Hölder or Lipschitz regularity coefficient of $f_i$ and $g_i$.
\end{lemma}
Our assumptions in Theorem \ref{theorem1.8}, in particular (6) and (7), guarantee that the map $F$ defined after \cite{cerrai2003stochastic}, hypothesis 4 satisfies the conditions outlined in (5.3), (5.4) and (5.5) of \cite{cerrai2003stochastic} so we should have a very similar estimate as in \cite{cerrai2003stochastic}, Lemma 5.4. However, as  \cite{cerrai2003stochastic}, Lemma 5.4 is stated on a different set of assumptions, we will prove Lemma \ref{1028} again under our assumption (namely, without assuming local Lipschitz continuity of the coefficients). The proof of Lemma \ref{1028} is given in Appendix \ref{appendixd}.

As a corollary of this estimate, we deduce that any solution to the SPDE will not blow up in finite time. For the given $T>0,$ we have $\mathbb{P}(\tau_\infty\leq T)=0$, where $\tau_\infty$ is the blowup time
    $$\tau_\infty:=\inf\{t\geq 0:\sup_{\xi\in[0,1],i\in[r]}|u_i(t,\xi)|=\infty\}.$$ Moreover, for any initial value $x\in H_0$, denote by $\tau_\infty^x$ the blowup time of the solution with initial value $x$, then we assume that for any $M>0$, we have uniformly
$$
\mathbb{P}(\inf_{|x|_{H_0}\leq M}\tau_\infty^x\leq T)=0.
$$

We are finally in the position to prove Theorem \ref{theorem1.8}.

\begin{proof}[\proofname\  of Theorem \ref{theorem1.8}] We start with the proof of weak uniqueness. Assume that we are given a weak mild solution $X_t$ on some probability space $(\Omega,\mathcal{F}_t,\mathbb{P})$. We find Lipschitz functions $(f_1^n,\cdots,f^n_r)$ and $(g_1^n,\cdots,g_r^n)$ such that we have the following uniform convergence on compact subsets of $\mathbb{R}$:
\begin{equation}\label{4.555}
    \sup_{t>0,\xi\in(0,1),i\in[r],|\sigma|\leq n} |f_i^n(t,\xi,\sigma)-f_i(t,\xi,\sigma)|+|g_i^n(t,\xi,\sigma)-g_i(t,\xi,\sigma)|:=\Delta_n\to 0,\quad n\to\infty.
\end{equation}
This Lipschitz approximation is easy to construct as we are dealing with real-valued functions on $\mathbb{R}^d$, so that convolution by a smooth mollifier suffices. Moreover, this approximation preserves the property that $g_i^n$ are uniformly non-degenerate, and that $f_i^n,g_i^n$ can have the same Hölder continuity coefficient as $f_i$ and $g_i$ (i.e. uniform in $n$ constants) for $\sigma\in\mathbb{R}^r:|\sigma|\leq n$ and for each $n$.

Then the SPDE with coefficients $f_i^n,g_i^n$ has a unique strong solution on this probability space. Denote this SPDE by 
\begin{equation}
    dX_t^n = AX_t^ndt+F^n(t,X_t^n)dt+G^n(t,X_t^n)dW_t, X_0^n=x.
\end{equation}
where $F^n$, $G^n$ are the abstract and compact notations for the coefficients $f_i^n,g_i^n.$ As before consider also the following auxiliary SPDE
\begin{equation}
    d\widetilde{X}_t^n = A\widetilde{X}_t^ndt+F^n(t,\widetilde{X}_t^n)dt+\lambda(X_t-\widetilde{X}_t^n)dt+G^n(t,\widetilde{X}_t^n)dW_t, X_0^n=x.
\end{equation}
for some $\lambda>0$ to be determined.
    Now we define two sequences of stopping times: (note the norm is the norm on the Banach space $H_0$)
    \begin{equation}
        \tau_n^1:=\inf\{t\geq 0: |X_t-\widetilde{X}_t^n|_{H_0}\geq 2\Delta_n\}, 
   \quad 
        \tau_m^2:=\inf\{t\geq 0: |X_t|_{H_0}|\geq m\}.
    \end{equation} By definition of blowup and the norm of $H_0$, one sees that \begin{equation}\label{upfrontupfront}\lim_{m\to\infty}\mathbb{P}(\tau_m^2\leq T)=0.\end{equation} Now we set $$\tau_{n,m}:=T\wedge\tau_n^1\wedge\tau_m^2.$$
    \textbf{Step 1: Linear contraction estimate.}
Before the stopping time $\tau_m^2\wedge T$, the process $X_t$ stays in a compact subset where the non-linearity $F$ has a finite Hölder constant $M_m$. The non-degeneracy condition on noise is also verified by assumption. By assumption, $F$ and $F^n$ are both locally Hölder continuous with continuity constant $M_n$, so one can bound
$$
    |F(t,X_t)-F^n(t,\widetilde{X}_t^n)|\leq M_m(2\Delta_n)^{\alpha}+\Delta_n,\quad t\leq \tau_{n,m}.
$$ 
Now we work as in the proof of Theorem \ref{theorem1.1}, using Theorem \ref{4.1fuc} instead of Theorem \ref{relabel}.
 Applying Theorem \ref{4.1fuc}, choosing $\lambda=(\Delta_K^n)^{\gamma -1}$, we obtain that for each $\eta\in(0,\eta_0)$ and $m'\in\mathbb{N}_+$, we may find a constant $C_m=C_{m,\eta,m',T}$ such that
\begin{equation}\label{cdddddomplementary}\mathbb{P}(\sup_{t\in[0,\tau_{n,m}]}|X_t-\widetilde{X}_t^n|_{H_0}>C_m (\Delta_n)^{1+\alpha-\gamma}+C_m(\Delta_n)^{\frac{1-\gamma}{2}\eta+\beta}R)\leq \frac{C_m}{R^{m'}},\quad R>0,\end{equation}
(We stress that the constant $C_m$ depends on $m$ because $f_i$ is only locally Hölder continuous, so that the Hölder constant depends on how large the solution grows, thus depends on $\tau_m^2$. In the following proof we will first send $n$ to infinity, and then send $m$
 to infinity), so that by choosing $R=(\Delta_n)^{-\chi}$ (with the choice of $\gamma,\eta,\chi\in(0,\frac{1-\gamma}{2}\eta+\beta-1]$ as in \eqref{condition2})  we can guarantee that for some $\chi_1',\chi_2'>0$, we have
\begin{equation}\label{}\mathbb{P}(\sup_{t\in[0,\tau_{n,m}]}|X_t-\widetilde{X}_t^n|_{H_0}>C_m (\Delta_n)^{1+\chi_1'})\leq C_m(\Delta_n)^{\chi_2'},\end{equation}
so that, given that $\Delta_n$ is sufficiently small for $n$ large, on the complement of the event stated in \eqref{cdddddomplementary}, one must have $\tau_{n,m}=T\wedge \tau_m^2,$ that is, 
\begin{equation}\label{111311131113}\mathbb{P}(\sup_{t\in[0,T\wedge \tau_m^2]}|X_t-\widetilde{X}_t^n|_{H_0}>C_m (\Delta_n)^{1+\chi_1'})\leq C_m (\Delta_n)^{\chi_2'}.\end{equation}
Since $H_0$ embeds continuously into $H$, this immediately implies a bound on $$\sup_{t\in[0,T\wedge\tau_m^2]}|X_t-\widetilde{X}_t^n|_H.$$

The above reasoning also suggests that for each $m$,
\begin{equation}\label{steagaagagrqegqgq}\mathbb{P}(\tau_n^1\geq T\wedge\tau_m^2)\to 0,\quad n\to\infty.\end{equation}

\textbf{Step 2: Estimates via Girsanov transform.} Consider an additional auxiliary SPDE defined via
\begin{equation}
    d\widetilde{\underline{X}}_t^n = A\widetilde{\underline{X}}_t^ndt+F^n(t,\widetilde{\underline{X}}_t^n)dt+\lambda(X_t-\widetilde{\underline{X}}_t^n)1_{t\leq \tau_{n,m}}dt+G^n(t,\widetilde{\underline{X}}_t^n)dW_t, \widetilde{\underline{X}}_0^n=x.
\end{equation} The solution to this SPDE is unique in the strong sense as the coefficients are all Lipschitz continuous.

Now we apply Girsanov transform, and by continuous embedding of $H_0$ into $H$ once more, we see that 
    \begin{equation}\label{variationiddddneq}
d_{TV}(\operatorname{Law}(X^n|_{[0,T]}),\operatorname{Law}(\widetilde{\underline{X}}^n|_{[0,T]}))\leq C T^{1/2}(\Delta_n)^\gamma.\end{equation}

\textbf{Step 3: Concluding proof of weak uniqueness.} Following the remaining steps in the proof of Theorem \ref{theorem1.1} one can show that for any real-valued bounded continuous function $E$ on $\mathcal{C}([0,T];H)$, we have
$$\lim\sup_{n\to\infty}
\left|\mathbb{E}[E(X|_{[0,T]}])-\mathbb{E}[E(X^n|_{[0,T]})]\right|\leq C\sup_x |E(x)|\mathbb{P}(\tau_{m}^2\leq T)\to 0,\quad m\to\infty,
$$
completes the proof of weak uniqueness.

More precisely, we need to combine the following three individual estimates: bounding $$\left|\mathbb{E}[E(X|_{[0,T]}])-\mathbb{E}[E(\widetilde{X}^n|_{[0,T]})]\right|$$ via \eqref{111311131113} and \eqref{upfrontupfront}; bounding  
$$\left|\mathbb{E}[E(X^n|_{[0,T]}])-\mathbb{E}[E(\widetilde{\underline{X}}^n|_{[0,T]})]\right|$$ via \eqref{variationiddddneq}; and finally bounding $$\left|\mathbb{E}[E(\widetilde{X}^n|_{[0,T]}])-\mathbb{E}[E(\widetilde{\underline{X}}^n|_{[0,T]})]\right|$$ via \eqref{steagaagagrqegqgq}.

\textbf{Step 4: Weak existence.}
To prove weak existence, one shall consider solutions $X_t^n,X_t^m$ driven by Lipschitz coefficients $F^n,F^m$ approximating $F$ in the above sense. By Lemma \ref{1028} and another estimate on the time increment of $X^n(t)$, applying Arzela-Ascoli, it is not hard to check the sequence $(X^n|_{[0,T]})_{n\geq 1}$ is tight on $\mathcal{C}([0,T];H_0)$. By Skorokhod embedding theorem and continuity of coefficients, taking a further limit leads to a weak mild solution to the SPDE with coefficients $F,G$. (For another approach, one may apply the coupling technique as in the previous steps to show that the law of $(X^n|_{[0,T]})_{n\geq 1}$ forms a converging sequence in $\mathcal{C}([0,T];H)$, and we upgrade this convergence to converging in $\mathcal{C}([0,T];H_0)$: this is the case because the contraction estimate is measured in $H_0$, giving a distance estimate of $\xi_1,\xi_2$ in $H_0$; and the total variation estimate can be turned into a coupling of two variables $\xi_2,\xi_3$ with high probability they are identical--see the next few paragraphs of this proof for what $\xi_1,\xi_2,\xi_3$ means. Thus now we have a coupling $\xi_1,\xi_3$ with a measurement of closeness using only the distance $H_0$ instead of $H$. As a result we have the sequence converging in law in $\mathcal{C}([0,T];H_0)$).

\textbf{Step 5: Continuity estimates.}
Now we prove the continuity estimate, which is similar to the procedure in Section \ref{sec3.145}. Having proved the solutions are unique in law, we skip the approximation procedures in the following argument and simply work with the equations themselves.

Fix two initial values $x,y\in H$. To couple the two processes
\begin{equation}\label{eqx123}dX_t=AX_t dt+F(t,X_t)dt+G(t,X_t)dW_t,\quad X_0=x\end{equation}
and 
\begin{equation}\label{eqy123}dY_t=AY_t dt+F(t,Y_t)dt+G(t,Y_t)dW_t,\quad Y_0=y,\end{equation}
we introduce a constant $\lambda>0$ and a stopping time $\tau$ whose value will be determined later, and we consider an auxiliary control process
\begin{equation}\label{controlprocegdss}
d\widetilde{Y}_t=A\widetilde{Y}_t dt+F(t,\widetilde{Y}_t)dt+\lambda(X_t-\widetilde{Y}_t)dt1_{t\leq\tau}+G(t,\widetilde{Y}_t)dW_t,\quad \widetilde{Y}_0=y.\end{equation}
Note that since the weak uniqueness of $\widetilde{Y}_t$ is proven, the solution $\widetilde{Y}_t$ does not depend on the stopping time $\tau$ whenever $t\leq\tau$.

We set $$\tau_{x,y}:=\inf\{t\geq 0:|X_t-\widetilde{Y}_t|_{H_0}\geq 2|x-y|_{H_0}\},$$
and set 
$\tau=\tau_{x,y}\wedge T\wedge \tau_m^2$.
Then we set $\lambda:=|x-y|_{H_0}^{\gamma -1}$ for some $\gamma\in(0,1)$ to be determined later. 

It follows from Proposition \ref{proposition2.3}, continuous embedding and the definition of $\tau$ that for any $T>0$ we can find $C$ such that 
\begin{equation}
    \label{controlgirsanovnew}
d_{TV}(\operatorname{Law}(Y|_{[0,T]}),\operatorname{Law}(\widetilde{Y}|_{[0,T]}))\leq CT^{1/2} |x-y|_{H_0}^\gamma.\end{equation}

Writing \eqref{eqx} and \eqref{eqy} in mild formulations, 
using the Hölder continuity of $F$ and $G$ we obtain the following deviation estimate thanks to Theorem \ref{4.1fuc}: for each $\eta\in(0,\frac{1}{2})$, we have that for each sufficiently large $m'\in\mathbb{N}_+$, we may find some fixed constant $M=M_{m',\eta}>0$ such that 
\begin{equation}\label{fifdddifi}
    \mathbb{P}(\sup_{0\leq t\leq \tau\wedge T}(|X_t-\widetilde{Y}_t|_{H_0}- e^{-\lambda t}|x-y|_{H_0})\geq  M\lambda^{-1}|x-y|_{H_0}^{\alpha}+M\lambda^{-\frac{1}{2}\eta}|x-y|_{H_0}^{\beta} R)\leq \frac{M}{R^{m'}},\quad R>0.
\end{equation}
By a suitable choice of $\lambda=|x-y|_{H_0}^{r-1}$, $R=|x-y|_{H_0}^{-\chi}$, where the constants $\gamma,\eta,\chi$ satisfy \eqref{730f}, \eqref{731f}, we finally obtain, for some $\chi_1'>0$, $\chi_2'>0$,
\begin{equation}\label{finalfinastste}
    \mathbb{P}\left(\sup_{0\leq t\leq\tau\wedge T} (|X_t-\widetilde{Y_t}|_{H_0}-e^{-|x-y|^{\gamma -1}t}|x-y|_{H_0})\geq C_m|x-y|_{H_0}^{1+\chi_1'}\right)\leq C_m |x-y|_{H_0}^{\chi_2'},
\end{equation} where the constant $C_m$ depends on $m$ again through the local Hölder constant of $f_i$. As all the exponents of $|x-y|$ in the left display are larger than one, by sample path continuity we deduce that on the complement of the event stated in the LHS of the last equation, one must have $\tau_{x,y}\geq T$, so that 
\begin{equation}\label{dgagagsag}
    \mathbb{P}\left(\sup_{0\leq t\leq\tau_m^2\wedge T} (|X_t-\widetilde{Y_t}|_{H_0}-e^{-|x-y|^{\gamma -1}t}|x-y|_{H_0})\geq C_m|x-y|_{H_0}^{1+\chi_1'}\right)\leq C_m |x-y|_{H_0}^{\chi_2'}.
\end{equation}

Define a distance $d$ on $H_0$ as follows:
$$d(x,y):=1\wedge |x-y|_{H_0},\quad x,y\in H.$$

Now we consider a coupling $(\xi_1,\xi_2)$ where $\xi_1$ has law $X_t$, $\xi_2$ has law $\widetilde{Y}_t$ such that the estimate \eqref{finalfinastste} is satisfied.
Consider another coupling $(\xi_2,\xi_3)$ where $\xi_2$ has law $\widetilde{Y}_t$ and $\xi_3$ has law $Y_t$, and such that by \eqref{controlgirsanovnew}, $$\mathbb{P}(\xi_2\neq \xi_3)\leq d_{TV}(\operatorname{Law}(Y|_{[0,T]});\operatorname{Law}(\widetilde{Y}|_{[0,T]}))\leq CT^{1/2}|x-y|_{H_0}^\gamma.$$
We may now form a coupling $(\xi_1,\xi_3)$ of $(X_t,Y_t),$ possibly on a different probability space, via the coupling $(\xi_1,\xi_2)$ and $(\xi_2,\xi_3)$ and obtain the following estimate (where we use $d$ to denote the Wasserstein distance defined by the distance function $d$, as in \eqref{wessagegag}):
\begin{equation}\label{whatgreatsup}\begin{aligned}
\mathbb{E}d(\xi_1,\xi_3)&\leq \mathbb{E} d(\xi_1,\xi_2)+\mathbb{P}(\xi_2\neq\xi_3)\\&\leq \mathbb{E}\left[d(\xi_1,\xi_2)1_{\{|\xi_1-\xi_2|_{H_0}\leq\frac{1}{2} |x-y|_{H_0}\}}1_{{\tau_m^2\geq T}\}}\right]
\\&+\mathbb{P}(|\xi_1-\xi_2|_{H_0}\geq \frac{1}{2}|x-y|_{H_0},\tau_m^2\geq T)\\&+\mathbb{P}(\xi_2\neq\xi_3)+\mathbb{P}(\tau_m^2\leq T).\end{aligned}
\end{equation} (When we bound the distance between $\xi_1$ and $\xi_2$, we separate the consideration in two cases: either $\tau_m^2\geq T$ and we apply \eqref{dgagagsag}, or $\tau_m^2\leq T$ and we use the trivial fact that the distance $d$ is upper bounded by one.)

The right hand side of \eqref{whatgreatsup} converges to $0$ as we set $|x-y|_{H_0}\to 0$ and set $m\to\infty.$ By definition of Wasserstein distance, we deduce that 
$$
d(\operatorname{Law}(X_t^x),\operatorname{Law}(X_t^y))\to 0,\quad |x-y|_{H_0}\to 0.
$$
This completes the proof of continuity in law with respect to the norm of $H_0$.
\end{proof}

\appendix
\section{Compactness of stochastic PDEs}\label{appendixA}
In the appendix we again assume that $A$ satisfies assumption $\mathbf{H}_1$.
The following proposition will be quite useful (see for example Proposition 8.4 of \cite{da2014stochastic}):

\begin{proposition}\label{prop1.3}
If $S(t),t>0$, are compact operators and $0<\frac{1}{p}<\eta\leq 1$, then the operator $G_\eta$
\begin{equation}\label{factorization}G_\eta f(t)=\int_0^t (t-s)^{\eta -1}S(t-s)f(s)ds,\quad f\in L^p((0,T);H),t\in[0,T]\end{equation}
is compact from $L^p((0,T);H)$ into $C([0,T];H).$
\end{proposition}
The proof consists in verifying that for each $t\in[0,T]$, the set $\{G_\eta f(t):|f|_p\leq 1\}$ is relatively compact in $H$, and that $G_\eta f(t),t\in[0,T]$ is uniformly continuous in $t$ with respect to the operator norm, i.e. $|G_\eta f(t)-G_\eta f(s)|\leq \epsilon$ if $|f|_p\leq 1$, $|t-s|\leq\delta$. As the concept of a set being relatively compact is equivalent to it being totally bounded, we may first find a finite subset  $\mathfrak{A}\subset[0,T]$ and a relatively compact subset $K_\mathfrak{A}$ of $H$ such that the unit ball of $L^p$ maps to $K_\mathfrak{A}$ under $G_\eta f(t)$ for each $t\in\mathfrak{A}$. Then we use the continuity of $G_\eta f(t)$ in operator norm to find a relatively compact subset $K\subset H$ such that $G_\eta f(t)$ maps the unit ball of $L^p$ into $K$, for all $t\in[0,T]$. 

The operator $G_\eta$ is quite useful because of the following factorization formula:
for any adapted process $\phi\in L^p(\Omega\times[0,T];\mathcal{L}(H_0,H)),$
\begin{equation}
    \int_0^t S(t-s)\phi(s)dW_s=\frac{\sin\pi\eta}{\pi} G_\eta Y(t),
\end{equation}
where 
\begin{equation}
    Y(t)=\int_0^t (t-s)^{-\eta} S(t-s)\phi(s)dW_s.
\end{equation}

Before we proceed, we note that the solution to the SPDE \eqref{evolutionequation} can be represented as
\begin{equation}\label{expansionrepresent}X_t=S(t)x+G_1 b(t)+\frac{\sin\pi\eta}{\pi}G_\eta \sigma(t),\end{equation}
with $b(t)=b(t,X_t)$
and $\sigma(t)=\int_0^t (t-s)^{-\eta} S(t-s)\sigma(s,X_s)dW_s.$

Under the linear growth assumption $\mathbf{H_5}$ and the semigroup assumption $\mathbf{H}_1$, one can find a constant $C_T>0$ such that 
\begin{equation}\label{momentestimate}
\sup_{t\in[0,T]}\mathbb{E}[|X_t|^p]\leq C_T(1+|x|^p),\end{equation} where $X$ is a solution to \eqref{evolutionequation} and $p\geq 2$. The proof can be found for example in Theorem 7.5 of \cite{da2014stochastic}. More precisely, we need to find Lipschitz approximations of the coefficients $b$ and $\sigma$, but as the resulting estimate is independent of the Lipschitz constant, we simily drop that approximation.

Now we can use \eqref{momentestimate} to estimate $b(t)$ and $\sigma(t)$. Via Young's inequality and BDG inequality,
\begin{equation}\label{youngbdg}\int_0^T \mathbb{E}[ |\sigma(s)|^p] ds\leq C_p\int_0^T\mathbb{E}\left[\left(\int_0^s (s-r)^{-2\eta}|S(s-r)\sigma(X_r)|_2^2dr\right)^{p/2}\right] ds,\end{equation}
which is further bounded, using assumption $\mathbf{H}_1$, estimate \eqref{momentestimate} and the linear growth assumption $\mathbf{H}_5$, by $C(1+|x|^p)$ where $C$ is some fixed constant. The estimate for $\int_0^T \mathbb{E}[|b(s)|^p]ds$ is similar. 

For any $0<p^{-1}<\delta\leq 1$ define 
\begin{equation}\label{blackboxs}\Lambda(R,\delta):=\{\omega\in C([0,T];H):\omega=G_\delta u,\quad \int_0^T |u(s)|^p ds<R\}.\end{equation}
Denote also by $$\Xi(R):=\{\omega\in C([0,T];H):\omega(t)=S(t)x+w_1(t)+w_2(t),w_1\in \Lambda(R,1), w_2\in \Lambda(R,\eta)\}.$$
Then $\Xi(R)$ is relatively compact in $C([0,T];H)$ for any $R>0$, thanks to Proposition \ref{prop1.3}. Note that the spaces $\Xi(R)$ and $\Lambda(R,\delta)$ are just function spaces for which the solution belongs to with high probability, so their definition does not involve coefficients $b$ and $\sigma$.

Moreover,  $\mathbb{P}(X|_{[0,T]}\notin \Xi(R))$ converges to $0$ as $R\to\infty$, thanks to the boundedness of  $\mathbb{E}[|\sigma(\cdot)|_{L^p([0,T];H)}]$ and $\mathbb{E}[|b(\cdot)|_{L^p([0,T];H)}]$, see \eqref{youngbdg}.

Now we can prove Proposition \ref{proposition1.2}.
\begin{proof} We deal with the case without the Burgers non-linearity $(-A)^{\vartheta}F$, and the case with such non-linearity is in the next Appendix.
For each $t\in[0,T]$ denote by $\Xi(R)_t$ the projection of $\Xi(R)$ onto the subspace $\{t\}\times H$, and identify $\{t\}\times H$ with $H$. Set $K:=\cup_{t\in[0,T]}\Xi(R)_t$. By the argument after the statement of Proposition \ref{prop1.3}, the set $K$ is relatively compact in $H$ for any $R>0$. It suffices to take $R$ sufficiently large to guarantee $\mathbb{P}(\theta_K\leq T)<\epsilon.$
\end{proof}

\section{Existence of weak mild solutions}\label{appendixB}
We again assume the semigroup satisfies assumption $\mathbf{H}_1$.

We first quote the following compactness result (see \cite{priola2022correction}, Proposition 11).

\begin{proposition}\label{propb1}
    Given $p>2$. The operator $Q:L^p([0,T];H)\to C([0,T];H)$
    $$Qf(t):=\int_0^t (-A)^{\vartheta}S(t-s)f(s)ds,\quad f\in L^p([0,T];H),t\in[0,T]$$
    is compact. 
\end{proposition}More precisely, \cite{priola2022correction}, Proposition 11 considered the case $\vartheta=\frac{1}{2}$, but the proof generalizes to all $\vartheta\in(0,1)$ without much difficulty.

Now we prove the existence of weak mild solutions needed in Theorem \ref{burgerstheorem1}. The proof is a generalization of the argument in \cite{priola2022correction}, Section 4, so we only give a sketch.

\begin{proof}
Denote by $\pi_m$ the orthogonal projection of $H$ onto the subspace spanned by the first $m$ eigenvalues $e_1,\cdots,e_m$. We write $A_m=A\circ \pi_m$. Using the weak existence result in \cite{gkatarek1994weak} (with straightforward generalization to time dependent coefficients), for each $m$ we can construct weak mild solutions to the SPDE 
\begin{equation}\label{approxSPDE}dX^m_t=A X^m_t dt+ b(t,X^m_t)dt+(-A_m)^{\vartheta}F(t,X^m_t)dt+\sigma(t,X^m_t)dW_t.\end{equation}
We can prove that for some $p>2$,
\begin{equation}\label{momentunis}\sup_{m\geq 1}\sup_{t\in[0,T]}\mathbb{E}|X^m_t|^p= C_p<\infty.\end{equation}
To obtain \eqref{momentunis} we need to control 
$$\left|\int_0^t (-A_m)^{\vartheta}S(t-s)F(s,X_s^m)ds\right|^p_H$$ via a Gronwall argument, which is done in Section 4 of \cite{priola2022correction}. We also need to control $$\left|\int_0^t S(t-s)
\sigma(s,X_s^m)dW_s\right|^p_H$$ via a Gronwall argument, which can be found in chapter 8 of \cite{da2014stochastic} (utilizing the mapping $G_\eta$ defined in \eqref{factorization}, Young's inequality and the assumption $\mathbf{H}_1$ on the semigroup $S(t)$.)

Now denote by $F_m:=\pi_m\circ F$, we learn from \eqref{momentunis}  and the linear growth property of $F$ that for any $T>0$, 
\begin{equation}
    \label{uniest12}
\sup_{m\geq 1}\mathbb{E}\int_0^T |F_m(t,X^m_t)|_H^pdt<\infty.\end{equation}

In the spirit of the decomposition \eqref{expansionrepresent}, we observe that the solution $X_t^m$ to \eqref{approxSPDE} can be reformulated as
$$X_t^m= S(t)x+G_1 b(X^m_\cdot)(t)+\frac{\sin\pi\eta}{\pi}G_\eta \sigma(X^m_\cdot)(t)+QF^n(X^m_\cdot)(t),$$
where $G_1$ and $G_\eta$ are defined in Appendix \ref{appendixA}. By $G_1b(\cdot,X_\cdot^m)(t)$ we mean that we consider the function $f:s\mapsto b(s,X_s^m)$, then compute $G_1(f)$, and evaluate the resulting function at time $t$. The same interpretation applies to $QF^m(\cdot,X_\cdot^m)$(t). The case of $G_\eta\sigma(\cdot,X_\cdot^m)(t)$  is slightly different, where we consider the function $f:s\mapsto (t-s)^{-\eta}S(t-s)\sigma(s,X_s^m)dW_s$.

Now for any $R>0$ define 
$$Q(R):=\{\omega\in C([0,T];H):\quad \omega=Q u,\quad \int_0^T |u(s)|^p ds<R\},$$
and define 
$$\begin{aligned}
\Xi(R):=\{&\omega\in C([0,T];H):\omega(t)=S(t)x+w_1(t)+w_2(t)+w_3(t),\\&w_1\in \Lambda(R,1), w_2\in \Lambda(R,\eta),w_3\in Q(R)\},\end{aligned}$$
where $\Lambda(R,1)$ and $\Lambda(R,\eta)$ are defined in \eqref{blackboxs}.
\end{proof}
By Proposition \ref{prop1.3} and \ref{propb1}, $\Xi(R)$ is compact in $C([0,T];H)$ for any $R>0$. By the uniform moment estimate \eqref{uniest12} and the corresponding one for $b(t,X_t^m)$ and $\sigma(t,X_t^m)$\footnote{For $\sigma(X_t^m)$ we need (uniform in $m$) $L^p$- estimates for the function $f^m:s\mapsto (t-s)^{-\eta}S(t-s)\sigma(X_s^m)dW_s$, which can be obtained via the infinite dimensional Burkholder inequality, the semigroup assumption $\mathbf{H}_1$ and the linear growth property of $\sigma$.}, we deduce that $$P((X_t^m)_{0\leq t\leq T}\notin \Xi(R))\to 0 \text{ as } R\to\infty,$$ uniformly over $m\geq 1$. By the Prokhorov theorem the laws of $(X_t^m)_{0\leq t\leq T}$ are tight in $C([0,T];H)$.

The next step is to use Skorokhod representation theorem to enlarge the probability space and take the limit of the tight sequence, using continuity of the coefficients and martingale representation theorem. The argument can be found in Section 2 of \cite{gkatarek1994weak}, so we omit the details.

\section{Verification of the Lyapunov condition}

\label{appendixC}
In this appendix we prove that when the functions $b$,$\sigma$ and $F$ are bounded, then $V(x):=|x|+1$ can serve as a Lyapunov function for the SPDEs we consider, as claimed in Remark \ref{remark111}. The general case of unbounded $b,\sigma,F$ of linear growth may also be considered, but we omit it for simplicity.

\begin{proof}We assume that the functions $b,\sigma$ and $F$ are bounded. Since $A$ is a negative operator with largest eigenvalue $-\lambda_1<0$, necessarily we have $\|S(t)\|_{op}\leq e^{-\lambda_1 t}$. Writing the SPDE in its mild formulation as follows:
\begin{equation}\begin{aligned}
    X_t&=S(t)x+\\&\int_0^t S(t-s)b(X_s)ds+\int_0^t (-A)^{\vartheta}S(t-s)F(X_s)ds+\int_0^t S(t-s)\sigma(X_s)dW_s.\end{aligned}
\end{equation}
For the first term, we use $|S(t)x|\leq e^{-\lambda_1 t}|x|$. For the remaining three terms in the second line, we use the boundedness of $b$, $F$ and $\sigma$, boundedness of the semigroup $S(t)$, as well as estimate \eqref{05.1} or \eqref{burgersusefulestimate}, to deduce that they are all bounded in expectation uniformly for all $t\in[0,T_0]$. Therefore taking expectation on both sides, we have
$$
\mathbb{E}|X_{t_0}|\leq e^{-\lambda_1 t_0}|x|+C_V,
$$ where $C_V>0$ is the almost sure upper bound for $\int_0^{t_0} S(t-s)b(X_s)ds+\int_0^{t_0} (-A)^{\vartheta}S(t-s)F(X_s)ds$, since the stochastic integral has zero expectation.
This verifies the claimed Lyapunov condition \eqref{lyapunov1} by giving the value of $c=1-e^{-\lambda_1 t_0}$ and $C_V>0$.
\end{proof}

\section{Proof of Lemma \ref{1028}}\label{appendixd}

Here we outline the proof of Lemma \ref{1028}.

We first need a notation on Banach space differentials. For any $x\in H_0$, $x$ is identified as an $r$-dimensional function $(x_i)_{i=1}^r$. We find $\xi_1,\cdots,\xi_r\in[0,1]$ such that $|x_i(\xi_i)|=|x_i|_{\mathcal{C}([0,1]}$. Let $H_0^*$ be the dual Banach space of $H_0$, and we find an arbitrary element $\delta\in H_0^*$ with norm 1. We define an element $\delta_x\in H_0^*$ such that for any $y\in H_0$ we have 
\begin{equation}\label{first1471}
    \langle \delta_x,y\rangle_{H_0}:=\begin{cases} \frac{1}{|x|_{H_0}}\sum_{i=1}^r x_i(\xi_i)y_i(\xi_i),\quad\text{if }x\neq 0,\\
\langle \delta,y\rangle_{H_0},\quad\text{if }x=0.
    \end{cases}
\end{equation}

We define the function $F(t,x)(\xi):=f(t,\xi,x(\xi))$ for each $x\in H_0$ and each $\xi\in[0,1]$. We can check that the following three inequalities are satisfied by $F$. First, for any $x,h\in L^{2m}([0,1];\mathbb{R}^r)$ (where $m$ was fixed in (6),polynomial growth \eqref{polyki}), we have
\begin{equation}\label{propertyf1}
    \langle F(t,x+h)-F(t,x),h\rangle_H\leq\Lambda(t)(1+|h|_H^2+|x|_H^2) 
\end{equation} for some $\Lambda\in L_{loc}^\infty[0,\infty)$.
To see this, we apply the dissipative condition (7) \eqref{dissipative336} for $k_i$ and apply the linear growth condition (5) \eqref{linearhi} for $f_i$. Combining both bounds leads to the desired estimate. Note that no Lipschitz continuity is used.

Second, we have that for some $\Phi\in L_{loc}^\infty[0,\infty)$ there holds
\begin{equation}\label{propertyf2}
|F(t,x)|_{H_0}\leq\Phi(t)(1+|x|_{H_0}^m),\quad x\in H_0.
\end{equation} This is immediate from (5) linear growth of $h_i$ and (6) polynomial growth of $k_i$. Finally, we have that there exists some $\Lambda\in L_{loc}^\infty[0,\infty)$ such that for all $x,h\in E$,
\begin{equation}\label{propertyf3}
\langle F(t,x+h)-F(t,x),\delta_h\rangle_{H_0}\leq\Lambda(t)(1+|h|_{H_0}+|x|_{H_0}).
\end{equation} This can be checked using the dissipative condition (7) for $k_i$ and the linear growth for $h_i$. Again we stress that we are not using local Lipschitz continuity of $h$ and $k$ at any place.

Now we prove Lemma \ref{1028}.

\begin{proof}[\proofname\ of Lemma \ref{1028}] Recall that in the statement of Lemma \ref{1028} we assumed $h_i$ and $k_i$ are both globally Lipschitz continuous. We shall be following closely the proof of \cite{cerrai2003stochastic}, Lemma 5.4 and the aim is to check that that proof continues to work here with no dependence on Lipschitz or Hölder constants of $h_i$ and $k_i$.

Let $u^x$ denote the solution to the original SPDE considered in Lemma \ref{1028}. Let $\Gamma(u^x)$ be the solution to 
\begin{equation}
    dv(t)=Av(t)dt+G(t,u^x(t))dW_t,\quad v(0)=0.
\end{equation}
Then $\Gamma(u^x)$ is the unique fixed point in $L^p(\Omega,C([0,T];H_0))$ of the mapping 
$$v(t)\mapsto \mathcal{S}(t-s)v(s)+\int_s^t\mathcal{S}(t-r)G(r,u^x(r))dW_r.$$ We denote by $\gamma(u^x)(t)=\int_s^t \mathcal{S}(t-r)G(r,u^x(r))dW_r$.

After a comparison argument as in \cite{cerrai2003stochastic} we can deduce that 
\begin{equation}\label{gammauxt}
|\Gamma(u^x)(t)|_{H_0}\leq c_s(t)\sup_{r\in[s,t]}|\gamma(u^x)(r)|_{H_0},
\end{equation} where $c_s(t)$ is some nonnegative function which is increasing in $t$ and satisfies $c_s(s)=0.$ This applies to all functions $c_s(t)$ in the following proof.

Now we subtract the random part and consider $v(t):=u^x(t)-\Gamma(u^x)(t)$. Then $v$ solves the following problem
$$
\frac{dv}{dt}(t)=Av(t)+F(t,v(t)+\Gamma(u^x)(t)),\quad v(s)=x.
$$
We assume $v$ is a strict solution (or can use approximation). Then we have the subdifferential inequalities
$$\begin{aligned}
&\frac{d}{dt}^{-}|v(t)|_{H_0}\leq\langle Av(t),\delta_{v(t)}\rangle_{H_0}+\langle F(t,v(t)+\Gamma(u^x)(t)),\delta_{v(t)}\rangle_{H_0}=\langle Av(t),\delta_{v(t)}\rangle_{H_0}\\&+
\langle F(t,v(t)+\Gamma(u^x)(t))-F(t,\Gamma(u^x)(t)),\delta_{v(t)}\rangle_{H_0}+\langle F(t,\Gamma(u^x(t)),\delta_{v(t)}\rangle_{H_0},\end{aligned}$$
where $\delta_{v(t)}$ is the element defined in \eqref{first1471}.

Using properties \eqref{propertyf1}, \eqref{propertyf2} and \eqref{propertyf3} satisfied by $F$, that $\langle Av(t),\delta_{v(t)}\rangle_{H_0}\leq 0$, and applying Young's inequality, we have 
$$
\frac{d}{dt}^{-}|v(t)|_{H_0}\leq\Lambda(t)|v(t)|_E+(\Lambda(t)+\Phi(t))\left(1+|\Gamma(u^x)(t)|_E^m\right).
$$
Then applying a comparison argument and using $u^x(t)=v(t)+\Gamma(u^x)(t)$, we get 
\begin{equation}\label{bothsdesequations}
|u^x(t)|_E\leq c_s(t)|x|_{H_0}+c_s(t)\left(1+\sup_{r\in[s,t]}|\Gamma(u^x)(r)|_{H_0}^m\right).
\end{equation}
    Since we assumed the diffusion coefficients are bounded from above \eqref{twosidedboundsg2}, we can use Theorem \ref{4.1fuc} (the $\lambda=0$ case) to estimate the stochastic integral $\gamma(u^x)(t)$ and get a bound 
    $$
\mathbb{E}\sup_{r\in[s,t]}|\gamma(u^x)(r)|^{pm}_{H_0}\leq c_{s,p}(t),\quad \forall p\in\mathbb{N}_+,
    $$ where the constant $c_{s,p}(t)$ further depends on $p$ and the supremum norm of $g_i$. Taking expectation on both sides of \eqref{bothsdesequations} and using \eqref{gammauxt}  we finally conclude that 
$$
\mathbb{E}|u^x(t)|_{H_0}^p\leq c_{s,p}(t)(1+|x|_{H_0})^p<\infty,
$$
    which completes the proof.
\end{proof}

\section*{Statements and Declarations
}

The author has no financial or non-financial interests to declare that are directly or indirectly related to the work submitted for publication.

\section*{Acknowledgement}

The author is very thankful to Professor Mickey Salins for offering enormous suggestions while the author was revising an earlier version of this paper. The author is also very thankful to the anonymous referee for pointing out numerous mistakes and impressions in the submitted manuscript. The majority of the work was completed when the author was affiliated with University of Cambridge as a doctoral student.

\printbibliography

\end{document}